\documentclass[11pt]{article}
\usepackage{srcltx}
\usepackage{eurosym}
\usepackage{mathtools}
\usepackage{amsmath}
\usepackage{amsfonts}
\usepackage{amssymb}
\usepackage{amsthm}
\usepackage{comment}
\allowdisplaybreaks
\usepackage{graphicx}
\usepackage{mathrsfs}
\usepackage{xcolor}
\usepackage{exscale}
\usepackage{latexsym}
\usepackage{authblk}
\newcommand{\overrightleft}[1]{%
  \overset{\stackrel{\rightarrow}{\leftarrow}}{#1}%
}
\newtheorem{assumption}{Assumption}

\usepackage[colorlinks,plainpages=true,pdfpagelabels,hypertexnames=true,colorlinks=true,pdfstartview=FitV,linkcolor=blue,citecolor=red,urlcolor=black]{hyperref}
\PassOptionsToPackage{unicode}{hyperref}
\PassOptionsToPackage{naturalnames}{hyperref}
\usepackage{enumerate}
\usepackage[shortlabels]{enumitem}
\usepackage{bookmark}
\usepackage{wasysym}
\usepackage[ddmmyyyy]{datetime}
\usepackage[margin=1in]{geometry}
\makeatletter
\g@addto@macro\normalsize{%
	\setlength\abovedisplayskip{4pt}
	\setlength\belowdisplayskip{4pt}
	\setlength\abovedisplayshortskip{4pt}
	\setlength\belowdisplayshortskip{4pt}
}
\numberwithin{equation}{section}
\everymath{\displaystyle}
\usepackage[capitalize,nameinlink]{cleveref}
\crefname{section}{Section}{Sections}
\crefname{subsection}{Subsection}{Subsections}
\crefname{condition}{Condition}{Conditions}
\crefname{hypothesis}{Hypothesis}{Conditions}
\crefname{lemma}{Lemma}{Lemmas}
\crefname{definition}{Definition}{Definitions}

\crefformat{equation}{\textup{#2(#1)#3}}
\crefrangeformat{equation}{\textup{#3(#1)#4--#5(#2)#6}}
\crefmultiformat{equation}{\textup{#2(#1)#3}}{ and \textup{#2(#1)#3}}
{, \textup{#2(#1)#3}}{, and \textup{#2(#1)#3}}
\crefrangemultiformat{equation}{\textup{#3(#1)#4--#5(#2)#6}}%
{ and \textup{#3(#1)#4--#5(#2)#6}}{, \textup{#3(#1)#4--#5(#2)#6}}%
{, and \textup{#3(#1)#4--#5(#2)#6}}

\Crefformat{equation}{#2Equation~\textup{(#1)}#3}
\Crefrangeformat{equation}{Equations~\textup{#3(#1)#4--#5(#2)#6}}
\Crefmultiformat{equation}{Equations~\textup{#2(#1)#3}}{ and \textup{#2(#1)#3}}
{, \textup{#2(#1)#3}}{, and \textup{#2(#1)#3}}
\Crefrangemultiformat{equation}{Equations~\textup{#3(#1)#4--#5(#2)#6}}%
{ and \textup{#3(#1)#4--#5(#2)#6}}{, \textup{#3(#1)#4--#5(#2)#6}}%
{, and \textup{#3(#1)#4--#5(#2)#6}}

\crefdefaultlabelformat{#2\textup{#1}#3}
\newtheorem{theorem} {Theorem}[section]

\newtheorem{lemma}[theorem]{Lemma}

\newtheorem{counter example}[theorem]{Counter Example}
\newtheorem{remark}[theorem] {Remark}
\newtheorem{definition}[theorem] {Definition}
\def\CC{{\rm \kern.24em \vrule width.02em height1.4ex depth-.05ex \kern-.26emC}}

\def\TagOnRight

\def\AA{{it I} \hskip-3pt{\tt A}}

\def\QQ{\rlap {\raise 0.4ex \hbox{$\scriptscriptstyle |$}} {\hskip -0.1em Q}}

\makeatletter
\newcommand{\vo}{\vec{o}\@ifnextchar{^}{\,}{}}
\makeatother
\def\YYint#1#2#3{{\setbox0=\hbox{$#1{#2#3}{\iint}$}
		\vcenter{\hbox{$#2#3$}}\kern-.50\wd0}}
\def\bm \sigmant#1{\mathchoice
	{\XXint\displaystyle\textstyle{#1}}%
	{\XXint\textstyle\scriptstyle{#1}}%
	{\XXint\scriptstyle\scriptscriptstyle{#1}}%
	{\XXint\scriptscriptstyle\scriptscriptstyle{#1}}%
	\!\int}
\def\XXint#1#2#3{{\setbox0=\hbox{$#1{#2#3}{\int}$}
		\vcenter{\hbox{$#2#3$}}\kern-.50\wd0}}

\makeatletter
\def\namedlabel#1#2{\begingroup
	\def\@currentlabel{#2}%
	\label{#1}\endgroup
}
\makeatother
\makeatletter
\newcommand{\rmh}[1]{\mathpalette{\raisem@th{#1}}}
\newcommand{\raisem@th}[3]{\hspace*{-1pt}\raisebox{#1}{$#2#3$}}
\makeatother

\newcounter{desccount}

\newcommand{\descref}[2]{\hyperref[#1]{\textnormal{\textcolor{black}{}\textcolor{blue}{ #2}\textcolor{black}{}}}}
\newcommand{\dref}[2]{\hyperref[#1]{\textcolor{black}{(}\textcolor{blue}{\bf #2}\textcolor{black}{)}}}
\newcommand{\be} {\begin{eqnarray}}
	\newcommand{\ee} {\end{eqnarray}}
\newcommand{\Bea} {\begin{eqnarray*}}
	\newcommand{\Eea} {\end{eqnarray*}}
\DeclareMathOperator{\lip}{Lip}

\newcommand{\gh}[1]{\left( #1\right)}

\newcommand{\dx}{\Delta x}
\newcounter{whitney}
\refstepcounter{whitney}

\newcounter{ineqcounter}
\refstepcounter{ineqcounter}
\makeatletter
\def\ps@pprintTitle{%
	\let\@oddhead\@empty
	\let\@evenhead\@empty
	\def\@oddfoot{}%
	\let\@evenfoot\@oddfoot}
\makeatother
\usepackage[doublespacing]{setspace}
\usepackage[titletoc,toc,page]{appendix}

\makeatletter
\newcommand{\refcheckize}[1]{%
	\expandafter\let\csname @@\string#1\endcsname#1%
	\expandafter\DeclareRobustCommand\csname relax\string#1\endcsname[1]{%
		\csname @@\string#1\endcsname{##1}\wrtusdrf{##1}}%
	\expandafter\let\expandafter#1\csname relax\string#1\endcsname
}
\makeatother
\refcheckize{\cref}
\refcheckize{\Cref}


\makeatletter
\newcommand{\mainsectionstyle}{%
	\renewcommand{\@secnumfont}{\bfseries}
	\renewcommand\section{\@startsection{section}{2}%
		\z@{.5\linespacing\@plus.7\linespacing}{-.5em}%
		{\normalfont\bfseries}}%
}
\makeatother
\usepackage{pgf,tikz}
\usetikzlibrary{arrows}
\usetikzlibrary{decorations.pathreplacing}
\usepackage{xpatch}
\xpatchcmd{\MaketitleBox}{\hrule}{}{}{}% remove first horizontal rule (above abstract)
\xpatchcmd{\MaketitleBox}{\hrule}{}{}{}% remoce second horizonral rule (below keywords)

\date{}
\usepackage{scalerel}
\makeatletter

\usepackage{subcaption}
\usepackage{hyperref}
\usepackage{caption}
\usepackage{subcaption}

\usepackage[utf8]{inputenc}
\allowdisplaybreaks
\usepackage{amsmath}
\usepackage{graphicx}
\usepackage{mathtools} 
\usepackage{hyperref}       
\usepackage{url}      
\usepackage{mathrsfs}
\usepackage{graphicx}
\usepackage{array}
\usepackage{color} 
\usepackage{tabularx}
\usepackage{amsmath,mathdots,amsthm}
\usepackage{amssymb}
\usepackage{amsfonts}
\usepackage{xcolor}
\usepackage{bm}
\usepackage{soul}
\usepackage{float}
\usepackage{cite}

\hypersetup{
	colorlinks=true,                          
	linkcolor=blue, % equation, section link color
	citecolor=blue, % bib color
	urlcolor=black  % url color if any
} 
\newtheorem{thm}{Theorem}[section]

\newtheorem{Lemma}[thm]{Lemma}

 \numberwithin{equation}{section}

\numberwithin{equation}{section}
\newcommand{\prn}[1]{\left( #1 \right)}

\newcommand{\rel}{^{\alpha, \tau}}

\newcommand{\hlf}{\frac{1}{2}}
\newcommand{\dt}{\Delta t}

\title{Energy consistent hyperbolic approximation for a class of fourth-order partial differential equations}

\hypersetup{
pdftitle={},
pdfauthor={},
pdfkeywords={},
}
\author[1]{Rahul Barthwal\thanks{\href{mailto:rahul.barthwal@mathematik.uni-stuttgart.de}{rahul.barthwal@mathematik.uni-stuttgart.de}}}
\author[2]{Firas Dhaouadi\thanks{\href{mailto:firas.dhaouadi@bordeaux-inp.fr}{firas.dhaouadi@bordeaux-inp.fr}}}
\author[1]{Christian Rohde\thanks{\href{mailto:christian.rohde@mathematik.uni-stuttgart.de}{christian.rohde@mathematik.uni-stuttgart.de}}}

\affil[1]{\footnotesize Institute of Applied Analysis and Numerical Simulation, University of Stuttgart\\

Pfaffenwaldring 57, 70569 Stuttgart, Germany}
\affil[2]{\footnotesize Institut de Mathématiques de Bordeaux, Université de Bordeaux, CNRS UMR 5251\\

Bordeaux INP, INRIA, F-33400, Talence, France}
\date{}

\begin{document}
\maketitle

\begin{abstract}
In this article, we propose a novel hyperbolic relaxation system for a general class of fourth-order nonlinear partial differential equations arising in the modelling of thin film flows or the phase separation of binary mixtures. The approximations are constructed to dissipate energies which recover the energy (Lyapunov) functional of the limit equation when the relaxation parameters vanish.

Using the relative energy framework, we prove the convergence of weak entropy solutions of the relaxation system to smooth solutions of the limit equation. We validate our analysis with a series of numerical examples for thin film equations and Cahn-Hilliard equations.
\end{abstract}
\medskip 
\section{Introduction}
In this article, we are interested in developing a lower-order approximate system for a general fourth-order nonlinear PDE of the form
\begin{align}\label{eq: main}
h_t + \nabla \cdot \left( \mathcal{M}(h)\nabla\big(\gamma \Delta h - \Pi(h)\big) \right) = 0
\quad \text{in } \Omega_T := \Omega \times (0,T) \subset \mathbb{T}^d\times\mathbb{R}_+ ,
\end{align}
where $\mathbb{T}^d$ denotes the $d$-dimensional torus. We supplement \eqref{eq: main} with initial data of the form
\begin{align}\label{initial_data_main}
    h(\cdot, 0)=h_0> 0.
\end{align}
The investigation of fourth-order nonlinear partial differential equations (PDEs) of the form \eqref{eq: main} has long attracted attention from physicists, engineers, and mathematicians because of their wide range of applications. In particular, such equations can model phenomena like thin-film and two-phase flows or plasticity dynamics; see e.g. \cite{elliott1996cahn, grun1995degenerate, beretta1995nonnegative} and references cited therein. In lubrication theory \cite{bertozzi1996lubrication, bertozzi1994lubrication}, $h$ describes the height of a viscous droplet spreading on a plain substrate, while in the case of the Cahn-Hilliard model for the phase separation of binary mixtures, $h$ is an indicator for the presence of the components. For the case of plasticity modelling, $h$ stands for the density of dislocations \cite{grun1995degenerate}. When the equation \eqref{eq: main} models fluid flow equations,  $\gamma>0$ is typically interpreted as the capillarity coefficient, $\mathcal{M}(h)$ as a mobility function, and $\Pi(h)$ as a disjoining-pressure (or effective chemical-potential) term describing the molecular interactions (e.g., van der
Waals forces acting between gas and solid) or thermal effects. In many cases of interest, $\Pi(h)$ can become singular at $h=0$. For thin film flows, the explicit form of the nonlinear mobility function $\mathcal{M}(h)$
depends upon the boundary condition at the liquid-solid interface. A no-slip condition
entails $\mathcal{M}(h)=h^3$, while the Navier slip conditions lead to $\mathcal{M}(h) = h^3 + \epsilon h^n, \epsilon >0, ~n\in (0, 3)$. Note that different choices of mobility and pressure functions can describe completely different physical mechanisms. For example, the choice of $\mathcal{M}(h)=h^3$ and $\Pi(h)=h$ leads to a thin film equation of the form
\begin{align}
    h_t=(h^3\nabla h)-\gamma(h^3\nabla \Delta h).
\end{align}
However, for $\mathcal{M}(h)=1$, it reduces to the Cahn-Hilliard equation with unit mobility, which has been analyzed extensively recently; see e.g. \cite{elliott1996cahn, dhaouadi2025first}.

An important property of the equation \eqref{eq: main} is its gradient-flow structure. Solutions of \eqref{eq: main} dissipate the Lyapunov (energy) functional
\begin{align}\label{energy}
F[h] = \int_{\Omega} \left( \frac{\gamma}{2}|\nabla h|^2 + W(h) \right)\, d\mathbf{x},\quad W'(h)=\Pi(h)
\end{align}
Relying on this energy structure of the equation \eqref{eq: main}, a well-established global existence theory of weak and strong solutions of \eqref{eq: main}-\eqref{initial_data_main} has been established for different choices of mobility functions; see e.g. \cite{elliott1996cahn, giacomelli1999fourth, gnann2018navier, giacomelli2014well, dai2016weak}. In particular, for non-degenerate mobilities with $d\leq 3$, local and global well-posedness of classical solutions in $H^s$-spaces for sufficiently large $s$ can be easily obtained for monotone pressures by utilizing classical parabolic theory; see \cite{bernis1990higher} for more details.

Substantial effort has also gone into understanding the qualitative behaviour (spreading, rupture, dewetting) of these flows; see, e.g., \cite{passo1998fourth,bernis1990higher} and the references cited therein. On the numerical side, positivity-preserving and energy-stable schemes have been designed for degenerate mobilities and singular pressures, ensuring mass conservation and monotone free-energy dissipation at the discrete level; see, e.g., \cite{zhornitskaya1999positivity,grun2000nonnegativity}. Collectively, the rich literature indicates that the mobility and pressure choices are not merely modelling details but determine existence theory, qualitative behaviour, and robust discretizations for \eqref{eq: main}-\eqref{initial_data_main}. Also, it was pointed out in the review paper by Bertozzi \cite{bertozzi1998mathematics} that, ``Even when the analytical solution is strictly positive, the solution of a generic scheme may become negative, especially when the grid is underresolved." This highlights the need for specially tailored numerical methods for equations of the form \eqref{eq: main} to preserve the qualitative dynamics of the solution.

One possible way to solve \eqref{eq: main}-\eqref{initial_data_main} numerically is to approximate it with lower-order systems, for which efficient numerical solvers are available. In recent years, much attention has been given to approximate higher-order PDEs with some lower-order approximate systems to understand the solution dynamics; see e.g., \cite{corli2014parabolic, corli2012singular}. In particular, hyperbolic relaxation systems have emerged as effective approximations of higher-order equations and have attracted substantial interest among researchers working in this area; see e.g., \cite{biswas2025traveling, dhaouadi2025first, giesselmann2025convergence, GKLR26, giesselmann2026justification}. However, most of these relaxation systems target models where the higher-order contribution is linear, so that no nonlinear product or state-dependent mobility appears in the highest derivatives, and also lack rigorous convergence analysis even with linear higher-order terms. In contrast, genuinely nonlinear fourth-order problems like \eqref{eq: main}, featuring (non-)degenerate mobilities and possibly singular pressures, remain comparatively underanalyzed within the relaxation framework. To address this gap, we propose a first-order hyperbolic relaxation system for \eqref{eq: main} in this article, generalizing the approach proposed in \cite{dhaouadi2025first}, to the case where non-constant mobility is present. More importantly, the relaxation system preserves the structural properties of the equation \eqref{eq: main}. In particular, the proposed system (see \eqref{hyperbolic_system}) preserves mass, mirrors the gradient-flow structure of \eqref{eq: main}, and preserves the energy structure of the underlying dynamics as the relaxation parameters vanish.
\medskip

Under some structural conditions on the mobility function $\mathcal{M}(h)$ and the pressure function $\Pi(h)$ (see Assumption (\ref{assumption.A}) below), we first prove that the Cauchy problem for \eqref{hyperbolic_system} is locally well-posed in the class of classical solutions. Furthermore, exploiting the energy structure of the hyperbolic system, we show that the Bregman distance between solutions of \eqref{hyperbolic_system} and solutions of the limit equation \eqref{eq: main} vanishes as the relaxation parameters tend to their limiting values. More precisely, the relative energy framework yields the convergence of weak entropy solutions of \eqref{hyperbolic_system} to a smooth solution of \eqref{eq: main} throughout the maximal interval of existence of the latter. For background on the relative entropy/energy method, we refer the reader to \cite{dafermos2005hyperbolic}.

We complement the theoretical analysis with a set of numerical experiments for thin-film and Cahn--Hilliard equations, demonstrating the accuracy and efficiency of the proposed relaxation system. Although Cahn--Hilliard-type equations are not covered by our analysis because the associated energy may be nonconvex, the numerical results indicate that \eqref{hyperbolic_system} nevertheless reproduces their nonlinear dynamics and captures the relevant solution structures with high accuracy. This provides evidence for the versatility and robustness of the proposed framework. To further illustrate its scope, we extend the relaxation approach to equations of the form \eqref{eq: main} with an additional first-order flux, for instance, from gravity or other nonlinear effects. The corresponding simulations show that our approximation approach accurately captures the resulting dynamics, including nonclassical shock waves, which is rare to find in the available literature on relaxation systems.

To make the following discussions precise, we now state the explicit assumptions on the mobility function $\mathcal{M}(h)$ and the pressure function $\Pi(h)$. They are assumed to satisfy the following conditions.
%\medskip\\
\begin{assumption}\label{assumption.A}
\leavevmode
Let $L, m_0, \delta>0$ be a constant. We have
\begin{enumerate}[label=($\mathcal{A}.\arabic*$),ref=$\mathcal{A}.\arabic*$]
  \item\label{A.1}The mobility function $\mathcal{M}:\mathbb{R}^d\to\mathbb{R}_+$ satisfies $\mathcal{M}\in C^\infty(\mathbb{R})$ with $\max_{h\in \Omega}\,\lvert \mathcal{M}'(h)\rvert < L$ and $\mathcal{M}(h)>1/m_0>0$ for all $h\in [0, \infty)$.
  \item\label{A.2} The pressure function $\Pi:\mathbb{R}^d\to\mathbb{R}$ satisfies $\Pi\in C^\infty(\mathbb{R})$ with $\max_{h\in \Omega}\lvert \Pi'(h)\rvert,\,\, \max_{h\in \Omega}\lvert \Pi''(h)\rvert  < L$ and $\Pi'(h)\geq \delta$ for all $h\in [0, \infty)$.
\end{enumerate}
\end{assumption}
\begin{remark}
Due to the monotonicity assumption (\ref{A.2}) on the pressure function, our analysis does not include typical Cahn-Hilliard type equations with non-monotone pressure.   
\end{remark}
The rest of the article is structured as follows. In Section \ref{sec: model}, we present the suggested relaxation \eqref{hyperbolic_system}, analyze its mathematical structure, and prove the local well-posedness of the Cauchy problem. 
Section \ref{sec: convergence} is devoted to proving the convergence of the solutions of the relaxation approximation to the solutions of \eqref{eq: main}-\eqref{initial_data_main} using the relative energy framework. In Section \ref{sec: numerics}, we provide several test cases validating the accuracy and robustness of the proposed relaxation system. Conclusions and outlooks are provided in Section \ref{sec: conclusions}.
\subsection*{Notations}\label{Notation}
Throughout the article, we use the following standard notations. We denote by $C^k(D)$ the space of $k$-times continuously differentiable functions defined in some domain $D\subset \mathbb{T}^d$ and $k\in \mathbb{N} \cup \{ \infty \}$. By $C^{\lip}_0(D)$, we denote the space of Lipschitz continuous functions having compact support. 
%$M_n(\mathbb{R})$ represents the linear space of $n\times n$ real-valued matrices. 
Further, for $p \in [1,\infty]$, we denote the Lebesgue space on $D$ as $L^p(D)$ and the $L^p(D)$-norm as ${\lVert \cdot \rVert}_{L^p(D)}$, which for a function $g \in L^p(D)$ is defined as
\begin{align*}
    {\lVert g\rVert}_{L^p(D)} = \biggl(\int_{D} |g(x)|^p \,dx\biggr)^{\frac{1}{p}} ~ \text{for}~p \in[1,\infty),
\end{align*}
and 
\[
{\lVert g\rVert}_{L^\infty(D)}
:=
\operatorname*{ess\,sup}_{x\in D}|g(x)|.
\]
We denote by $H^k(D)$ the Sobolev space of order $k$ consisting of 
square-integrable functions whose weak derivatives up to order $k$ are also square-integrable, that is
\[
    H^k(D)
    := \bigl\{ g \in {L}^{2}(D) : \partial^m g \in {L}^{2}(D)
        \text{ for all multi-indices}~ m \text{ with } |m| \le k \bigr\},
\]
equipped with the norm
\[
    {\lVert g \rVert}_{H^k(D)}^2
    := \sum_{|m| \le k} {\lVert \partial^m g \rVert}_{{L}^{2}(D)}^2,
\]
where $\partial^m g$ denotes the weak derivative of $g$ of order $m$. \\
Finally, for a Banach space $X$ and $T>0$, we denote by $L^p(0, T;X)$, the Bochner space of measurable functions $g:(0,T)\to X$ with an induced norm
\[
    {\lVert g \rVert}_{L^p(0,T;X)}
    = \biggl( \int_0^T {\lVert u(t) \rVert}_{X}^p \,\mathrm{d}t \biggr)^{1/p},
\]
for $p \in [1,\infty)$, while for $p = \infty$
\[
{\lVert g \rVert}_{L^{\infty}(0,T;X)}
    = \operatorname*{ess\,sup}_{t\in[0, T]} {\lVert g(t) \rVert}_{X}.
\]
%%%%%%%%%%%%%%%%%%%%%%%%%%%%%%%%%%%%%%%%%%%%%%%%%%%
\section{The relaxation approximation and the local well-posedness of the Cauchy problem}\label{sec: model}
In this section, we suggest a first-order relaxation system to approximate the limit equation \eqref{eq: main}. Our approach relies on the ideas from \cite{dhaouadi2025first,corli2012singular, corli2014parabolic}. We show that the relaxation system is hyperbolic and prove the well-posedness of the associated Cauchy problem.
\subsection{Formulation of the relaxation approximation}
Based on the relaxation approaches for diffusive-dispersive equations in \cite{corli2012singular, corli2014parabolic}, we propose to relax the equation \eqref{eq: main} by a nonlocal equation of the form 
\begin{align}\label{first_relaxation}
    h^{\alpha}_t=\nabla\cdot\big(\mathcal{M}(h^{\alpha})\nabla \big(\Pi(h^{\alpha})-\alpha (\mathcal{K}^{\alpha} \ast h^{\alpha}-h^{\alpha})\big)\big)\quad \mathrm{in}\quad \Omega_T,
\end{align}
which depends on the parameter $\alpha>0$. In \eqref{first_relaxation} the symbol ``$\ast$" denotes spatial convolution and $\mathcal{K}^{\alpha}: \mathbb{R}^d\mapsto \mathbb{R}$ is a symmetric kernel function, which satisfies 
\[
\int_{\mathbb{R}^d} \mathcal{K}^{\alpha}(\mathbf{x})\, d\mathbf{x}=1.
\]
Choosing the kernel $\mathcal{K}^\alpha$ as the Green function \cite{corli2012singular} of the screened Poisson equation for an additional unknown $\phi^{\alpha}:\Omega_T\rightarrow \mathbb{R}$ such that $\phi^{\alpha}=\mathcal{K}^{\alpha}\ast h^{\alpha}$,  
one can rewrite equation \eqref{first_relaxation} into the equivalent local parabolic–elliptic form
\begin{equation}\label{local_relaxation}
\begin{aligned}
h^{\alpha}_t&=\nabla\cdot(\mathcal{M}(h^{\alpha})\nabla(\Pi(h^{\alpha})+\alpha(h^{\alpha}-\phi^{\alpha}))),\\
   -\gamma \Delta \phi^{\alpha}+\alpha \phi^{\alpha}&=\alpha h^{\alpha}
\end{aligned}
\quad \mathrm{in}\quad \Omega_T.
\end{equation}
Similar to \cite{barthwal2025relaxation, dhaouadi2025first}, we add temporal dynamics to the variable $\phi$ using a linear wave operator. Precisely, we consider the system
\begin{equation}\label{wave_system}
\begin{aligned}
h^{\alpha}_t&=\nabla\cdot(\mathcal{M}(h^{\alpha})\nabla(\Pi(h^{\alpha})+\alpha(h^{\alpha}-\phi^{\alpha}))),\\
   \beta \phi^{\alpha}_{tt}-\gamma \Delta \phi^{\alpha}&=\alpha(h^{\alpha}-\phi^{\alpha}),
\end{aligned}
\end{equation}
where $\beta$ is a parameter depending on $\alpha$ such that $\beta(\alpha)\rightarrow 0$ as $\alpha\rightarrow \infty$. We will prove in Section \ref{sec: convergence} that the solutions $(h^\alpha, \phi^\alpha)$ to the Cauchy problem for \eqref{wave_system} converges to $(h, h)$, with $h$ being the solution to \eqref{eq: main}-\eqref{initial_data_main}.

Note that the wave equation for $\phi^{\alpha}$ can be converted into an equivalent first-order form as long as boundary-compatible conditions are prescribed. The first equation in \eqref{wave_system} still possesses second-order terms and therefore, we use an additional relaxation approach of the Cattaneo type \cite{cattaneo1948sulla} to approximate the first equation of \eqref{wave_system} with a relaxation parameter $\tau$ such that
\begin{equation}\label{q_relaxation}
\begin{aligned}
    h^{\alpha, \tau}_t+\nabla\cdot \mathbf{q}^{\alpha, \tau}&=0,\\
    \tau \mathbf{q}^{\alpha, \tau}_t+\nabla\big(\Pi(h^{\alpha, \tau})+\alpha(h^{\alpha, \tau}-\phi^{\alpha, \tau})\big)&=-\dfrac{\mathbf{q}^{\alpha, \tau}}{\mathcal{M}(h^{\alpha, \tau})}.  
\end{aligned}
\end{equation}
Formally speaking, we recover the second-order system \eqref{wave_system} when $\tau\rightarrow 0$. Now combining the systems \eqref{wave_system} and \eqref{q_relaxation} and using the definitions
\[w^{\alpha, \tau}=\phi^{\alpha, \tau}_t,\quad \quad \mathbf{p}^{\alpha, \tau}=\nabla \phi^{\alpha, \tau},\] 
one obtains the first-order system
\begin{equation*}\label{hyperbolic_system_new}
\begin{aligned}
    h^{\alpha, \tau}_t+\nabla\cdot \mathbf{q}^{\alpha, \tau}&=0,\\
     \phi^{\alpha, \tau}_t&=w^{\alpha, \tau},\\
    \tau \mathbf{q}^{\alpha, \tau}_t+\nabla\big(\Pi(h^{\alpha, \tau})+\alpha(h^{\alpha, \tau}-\phi^{\alpha, \tau})\big)&=-\dfrac{\mathbf{q}^{\alpha, \tau}}{\mathcal{M}(h^{\alpha, \tau})},\\
    \beta w^{\alpha, \tau}_t-\gamma \nabla\cdot \mathbf{p}^{\alpha, \tau}&=\alpha(h^{\alpha, \tau}-\phi^{\alpha, \tau}),\\
    \mathbf{p}^{\alpha, \tau}_t-\nabla w^{\alpha, \tau}&=0
\end{aligned}\quad \mathrm{in}\quad \Omega_T.
\end{equation*}
With the variable $\psi^{\alpha, \tau}=\alpha(h^{\alpha, \tau}-\phi^{\alpha, \tau})$, we obtain a system of the form
\begin{equation}\label{hyperbolic_system}
\begin{aligned}
     h^{\alpha, \tau}_t+\nabla\cdot \mathbf{q}^{\alpha, \tau}&=0,\\
     \dfrac{1}{\alpha}\psi^{\alpha, \tau}_t+\nabla\cdot \mathbf{q}^{\alpha, \tau}&=-w\rel,\\
     \tau\mathbf{q}^{\alpha, \tau}_t+\nabla(\Pi(h^{\alpha, \tau})+\psi^{\alpha, \tau})&=-\dfrac{\mathbf{q}^{\alpha, \tau}}{\mathcal{M}(h^{\alpha, \tau})},\\
    \beta w^{\alpha, \tau}_t-\gamma \nabla\cdot \mathbf{p}^{\alpha, \tau}&=\psi^{\alpha, \tau},\\
    \mathbf{p}^{\alpha, \tau}_t-\nabla w^{\alpha, \tau}&=0
\end{aligned} \quad \mathrm{in}\quad \Omega_T.
\end{equation}
In what follows, we treat the system \eqref{hyperbolic_system} as our main relaxation approximation. In \eqref{hyperbolic_system}, $\alpha, \beta$ and $\tau$ are positive numbers such that $\tau\rightarrow 0$ and $\beta:=\beta(\alpha)\rightarrow 0$ as $\alpha\rightarrow \infty$. In the stiff relaxation limit, we expect that the variable $\psi^{\alpha, \tau}$ approximates $-\gamma \Delta h$, $\mathbf{q}^{\alpha, \tau}$ approximates the nonlinear flux $-\mathcal{M}(h)\nabla(\Pi(h)-\gamma \Delta h)$, $w^{\alpha, \tau}$ approximates $h_t$ and $\mathbf{p}^{\alpha, \tau}$ approximates $\nabla h$. This indicates that the system \eqref{hyperbolic_system} is a consistent approximation of \eqref{eq: main}.

For the discussion to be followed, we define the state space for the system \eqref{hyperbolic_system} as follows.
\begin{align}\label{statespace}
    \mathcal{U}=\{(h, \psi, \mathbf{q}, w, \mathbf{p})^\top \in \mathbb{R}^{2d+3}: h>0\}
\end{align}
We complement the system \eqref{hyperbolic_system} for the unknown $\mathbf{U}\rel=(h\rel, \psi\rel, \mathbf{q\rel}, w\rel, \mathbf{p\rel})^\top$ with well-prepared initial data of the form
\begin{equation}\label{initial_data_hyp_relaxation}
\mathbf U\rel_0=\mathbf U\rel(\cdot,0)=
\bigl(h_0,-\gamma\Delta h_0,-\mathcal M(h_0)\nabla(\Pi(h_0)-\gamma\Delta h_0),\partial_t h(\cdot,0),\nabla h_0\bigr)^\top,
\end{equation}
where $h_0$ is defined in \eqref{initial_data_main}. In Section \ref{sec: convergence}, we explicitly prove that weak entropy solutions of the system \eqref{hyperbolic_system} (see Definition \ref{entropy_soln}) with well-prepared initial data of the form \eqref{initial_data_hyp_relaxation} converge to a sufficiently regular solution of the Cauchy problem \eqref{eq: main}-\eqref{initial_data_main}.
\subsection{Hyperbolicity of the relaxation approximation \eqref{hyperbolic_system}}
We now study the hyperbolicity of the system
\eqref{hyperbolic_system}. For the sake of simplicity, we omit the superscripts
$(\alpha,\tau)$. Let $\mathbf n\in\mathbb S^{d-1}$ be an arbitrary unit vector, and consider an orthonormal basis $\{\mathbf n,\mathbf t_2,\ldots,\mathbf t_d\}$ of $\mathbb R^d$. 

For the unknowns $\mathbf{q}$ and $\mathbf{p}$, we define the normal components by
\[
q_1=\mathbf q\cdot\mathbf n,
\qquad
p_1=\mathbf p\cdot\mathbf n,
\]
and the tangential components by
\[
q_j=\mathbf q\cdot\mathbf t_j,
\qquad
p_j=\mathbf p\cdot\mathbf t_j,
\qquad j=2,\ldots,d.
\]
Thus, we can decompose $\mathbf{q}$ and $\mathbf{p}$ as
\[
\mathbf q
=
q_1\mathbf n+\sum_{j=2}^d q_j\mathbf t_j,
\qquad
\mathbf p
=
p_1\mathbf n+\sum_{j=2}^d p_j\mathbf t_j.
\]
We then introduce the reordered state vector
\begin{align}\label{U_vector}
\mathbf U
=
\bigl(
h,\psi,q_1,w,p_1,
q_2,\ldots,q_d,p_2,\ldots,p_d
\bigr)^{\mathsf T}
\in\mathbb [0, \infty)\times R^{2d+2}.
\end{align}
The variables $q_1$ and $p_1$ are the normal components in the
direction $\mathbf n$, whereas $q_j$ and $p_j$, $j=2,\ldots,d$, are
the components tangent to the hyperplane $\mathbf n^\perp$. We then prove the hyperbolicity of the system \eqref{hyperbolic_system} in the following lemma.
\begin{lemma}
Let $\alpha, \beta,\gamma, \tau>0$ such that $\alpha+\Pi'(h)>0$ and $\alpha+\Pi'(h)> \dfrac{\gamma \tau}{\beta}$ for all $h\in [0, \infty)$. \\
Then, the system \eqref{hyperbolic_system} is hyperbolic with $2d+3$ real eigenvalues given by
\begin{equation}\label{eigen_values}
\begin{aligned}
\lambda_1=-\sqrt{\frac{\alpha+\Pi'(h)}{\tau}},\,\,\lambda_2=-\sqrt{\frac{\gamma}{\beta}},\,\,
\lambda_{3}=\ldots=\lambda_{2d+1}=0,\,\,\lambda_{2d+2}
=\sqrt{\frac{\gamma}{\beta}},\,\,\lambda_{2d+3}=\sqrt{\frac{\alpha+\Pi'(h)}{\tau}}
.
\end{aligned}
\end{equation}
Moreover, the system \eqref{hyperbolic_system} possesses a full set of linearly independent eigenvectors and thus is a strongly hyperbolic system.
\end{lemma}
\begin{proof}
Since the source term in \eqref{hyperbolic_system} does not affect the hyperbolicity of the system \eqref{hyperbolic_system}, we focus only on the principal part of the system for the vector $\mathbf{U}$ as defined in \eqref{U_vector}. The principal part is given by
\[
\mathbf{U}_t+\mathbf{A(U)}\mathbf{U}_{x}=0,
\]
where
\[
\mathbf A(\mathbf U)
=
\begin{pmatrix}
0&0&1&0&0\\[0.2em]
0&0&\alpha&0&0\\[0.2em]
\dfrac{\Pi'(h)}{\tau}&\dfrac{1}{\tau}&0&0&0\\[0.6em]
0&0&0&0&-\dfrac{\gamma}{\beta}\\[0.6em]
0&0&0&-1&0
\end{pmatrix}
\oplus
\mathbf{0}_{(2d-2)\times(2d-2)}.
\]
Here, $\oplus$ denotes the direct sum of the $5\times 5$ nonzero block corresponding to $(h, \psi, q_1, w, p_1)^\top$ and the ${{(2d-2)}\times {(2d-2)}}$ zero block corresponding to $(q_2, \ldots, q_d, p_2, \ldots, p_d)^\top$. A straightforward calculation using the properties of direct sums then leads us to the eigenvalues of $\mathbf{A(U)}$ as noted in \eqref{eigen_values}.

Clearly, under the assumption that $\alpha+\Pi'(h)>0$ and $\alpha+\Pi'(h)> \dfrac{\gamma \tau}{\beta}$, all $2d+3$ eigenvalues are real. Among these $2d+3$ eigenvalues, $\lambda_1, \lambda_2, \lambda_{2d+2}$ and $\lambda_{2d+3}$ are distinct while $\lambda=0$ is a repeated eigenvalue of algebraic multiplicity $2d-1$. In what follows, we prove that the eigenspace corresponding to $\lambda=0$ is of dimension $2d-1$. This is then enough to prove that the matrix $\mathbf{A(U)}$ has a full set of eigenvectors. We now compute the eigenvectors corresponding to zero eigenvalues explicitly. First, the eigenvector corresponding to the zero eigenvalue from the nonzero part of the matrix $\mathbf{A(U)}$ reads as $R_0=\left(1,0,0,0,-\Pi'(h),\mathbf 0_{2d-2}\right)^\top$.

The remaining eigenvectors are those associated with the
tangential components of $\mathbf q$ and $\mathbf p$ and thus come from the zero block of the matrix $\mathbf{A(U)}$. Let $\{\mathbf e_1,\ldots,\mathbf e_{2d-2}\}$ denote the orthonormal basis of the subspace $\mathbb R^{2d-2}$. Then, the eigenvectors corresponding to the remaining zero eigenvalues can be easily obtained using the properties of the direct sum. Precisely, the eigenvectors coming from the zero block read as $R_k^{\perp}=(0, 0, 0, 0, 0, \mathbf{e}_k)$ for $k=1, 2, \ldots, 2d-2$. Clearly, the set of eigenvectors $\{R_k^{\perp}, \, 1\leq k\leq 2d-2\}$ is linearly independent of $R_0$. Thus the geometric multiplicity corresponding to the zero eigenvalue is $2d-1$. The lemma is proved.
\end{proof}
\subsection{A Symmetrizer for the relaxation approximation \eqref{hyperbolic_system} and the local wellposedness}
In this section, we prove that the system \eqref{hyperbolic_system} can be symmetrized using a symmetric positive definite symmetrizer and thus belongs to the class of Friedrichs symmetrizable systems. We prove this in the following lemma.
\begin{lemma}\label{symmetrizer}
Suppose that there exists a number
\(\kappa>0\) such that $\alpha+\Pi'(h)\geq \kappa>0$. Then the first-order system \eqref{hyperbolic_system} is a Friedrichs-symmetrizable system. In particular, the system \eqref{hyperbolic_system} can be symmetrized using a symmetric positive definite matrix.
\end{lemma}
\begin{proof}
For $\mathbf U
=
\left(h,\psi,\mathbf q,w,\mathbf p\right)^\top\in \mathcal{U},$ the system can be written as
\[
\mathbf U_t
+\sum_{j=1}^d
\mathbf A_j(\mathbf U)\partial_{x_j}\mathbf U
=
\mathbf S(\mathbf U),
\]
where
\[
\mathbf A_j(\mathbf U)
=
\begin{pmatrix}
0&0&\mathbf e_j&0&\mathbf 0_d\\
0&0&\alpha\mathbf e_j&0&\mathbf 0_d\\
\dfrac{\Pi'(h)}{\tau}\mathbf e_j^\top&
\dfrac1\tau\mathbf e_j^\top&
\mathbf 0_{d\times d}&
\mathbf 0_d^\top&
\mathbf 0_{d\times d}\\
0&0&\mathbf 0_d&0&
-\dfrac{\gamma}{\beta}\mathbf e_j\\
\mathbf 0_d^\top&\mathbf 0_d^\top&
\mathbf 0_{d\times d}&
-\mathbf e_j^\top&
\mathbf 0_{d\times d}
\end{pmatrix}.
\]
To symmetrize the system \eqref{hyperbolic_system} by a symmetric positive definite matrix $\mathbf{A_0(U)}$, define for numbers $c, e>0$ the matrix
\[
\mathbf A_0(\mathbf U)
=
\begin{pmatrix}
a(h)&b&\mathbf 0_d&0&\mathbf 0_d\\
b&c&\mathbf 0_d&0&\mathbf 0_d\\
\mathbf 0_d^\top&\mathbf 0_d^\top&
e\mathbf I_{d\times d}&
\mathbf 0_d^\top&\mathbf 0_{d\times d}\\
0&0&\mathbf 0_d&1&\mathbf 0_d\\
\mathbf 0_d^\top&\mathbf 0_d^\top&
\mathbf 0_{d\times d}&
\mathbf 0_d^\top&
\dfrac{\gamma}{\beta}\mathbf I_{d\times d}
\end{pmatrix}.
\]
Note that the matrix $\mathbf A_0(\mathbf U)$ has the block structure
\[
\mathbf A_0(\mathbf U)
=
\begin{bmatrix}
\mathbf S_1(\mathbf U)
&
\mathbf 0_{(d+2)\times(d+1)}
\\
\mathbf 0_{(d+1)\times(d+2)}
&
\mathbf S_2(\mathbf U)
\end{bmatrix},
\]
where
\[
\mathbf S_1(\mathbf U)
=
\begin{bmatrix}
a(h)&b&\mathbf 0_d\\
b&c&\mathbf 0_d\\
\mathbf 0_d^\top&\mathbf 0_d^\top&
e\mathbf I_{d\times d}
\end{bmatrix},
\qquad
\mathbf S_2(\mathbf U)
=
\begin{bmatrix}
1&\mathbf 0_d\\
\mathbf 0_d^\top&
\dfrac{\gamma}{\beta}\mathbf I_{d\times d}
\end{bmatrix}.
\]
Note that the subblock $\mathbf{S}_2(\mathbf{U})$ is already a symmetric positive definite matrix, which symmetrizes the wave operator subblock of the system \eqref{hyperbolic_system}, and thus we proceed to obtain the conditions for the positive definiteness of the other subblock $\mathbf{S}_1(\mathbf{U})$ such that $\mathbf{S}_1(\mathbf{U})$ symmetrizes the $(h,\psi,\mathbf q)$-subsystem.

For a fixed $j\in\{1,\ldots,d\}$, the principal matrix of the
$(h,\psi,\mathbf q)$-subsystem in the $x_j$-direction is
\[
\mathbf A_{1,j}(\mathbf U)
=
\begin{bmatrix}
0&0&\mathbf e_j\\
0&0&\alpha\mathbf e_j\\
\dfrac{\Pi'(h)}{\tau}\mathbf e_j^\top&
\dfrac1\tau\mathbf e_j^\top&
\mathbf 0_{d\times d}
\end{bmatrix}.
\]
Multiplying $\mathbf A_{1,j}(\mathbf U)$ by
$\mathbf S_1(\mathbf U)$, we obtain
\begin{align*}
\mathbf S_1(\mathbf U)\mathbf A_{1,j}(\mathbf U)
&=
\begin{bmatrix}
a(h)&b&\mathbf 0_d\\
b&c&\mathbf 0_d\\
\mathbf 0_d^\top&\mathbf 0_d^\top&
e\mathbf I_{d\times d}
\end{bmatrix}
\begin{bmatrix}
0&0&\mathbf e_j\\
0&0&\alpha\mathbf e_j\\
\dfrac{\Pi'(h)}{\tau}\mathbf e_j^\top&
\dfrac1\tau\mathbf e_j^\top&
\mathbf 0_{d\times d}
\end{bmatrix}
\\
&=
\begin{bmatrix}
0&0&\bigl(a(h)+\alpha b\bigr)\mathbf e_j\\
0&0&\bigl(b+\alpha c\bigr)\mathbf e_j\\
\dfrac{e\Pi'(h)}{\tau}\mathbf e_j^\top&
\dfrac e\tau\mathbf e_j^\top&
\mathbf 0_{d\times d}
\end{bmatrix}.
\end{align*}
Clearly, the matrix $\mathbf{S_1(U)}\mathbf{A_1(U)}$ is symmetric iff $a(h)+b\alpha=\dfrac{e\Pi'(h)}{\tau}$ and $b+c\alpha=\dfrac{e}{\tau}$ or equivalently $a(h)=\alpha^2 c+\dfrac{e}{\tau}(\Pi'(h)-\alpha)$ hold. Moreover, the matrix $\mathbf{S_1(U)}$ is positive definite iff $a(h)>0$ and $a(h)c-b^2=\dfrac{e}{\tau^2}(c\tau(\alpha+\Pi'(h))-e)>0$ hold. These two inequalities together imply that
\[
\dfrac{e}{c\tau}<\min\bigg\{\dfrac{\alpha^2}{\alpha-\Pi'(h)}, (\alpha+\Pi'(h))\bigg\}.
\]
Thus, the necessary condition for $\mathbf{S}_1(\mathbf{U})$ to be positive definite reduces to
\[
\dfrac{e}{c\tau}<\alpha+\Pi'(h).
\]
Choosing $c=1$, $\alpha+\Pi'(h)\geq \kappa>0$ and $\dfrac{e}{\tau}=\dfrac{\kappa}{2}$, the matrix $\mathbf{A_0(U)}$ then becomes
\[
\mathbf A_0(\mathbf U)
=
\begin{bmatrix}
\alpha^2+\dfrac{\kappa}{2}\bigl(\Pi'(h)-\alpha\bigr)
&
\dfrac{\kappa}{2}-\alpha
&
\mathbf 0_d&0&\mathbf 0_d
\\[0.8em]
\dfrac{\kappa}{2}-\alpha
&
1
&
\mathbf 0_d&0&\mathbf 0_d
\\[0.8em]
\mathbf 0_d^\top&\mathbf 0_d^\top&
\dfrac{\tau\kappa}{2}\mathbf I_{d\times d}
&
\mathbf 0_d^\top&\mathbf 0_{d\times d}
\\[0.8em]
0&0&\mathbf 0_d&1&\mathbf 0_d
\\[0.3em]
\mathbf 0_d^\top&\mathbf 0_d^\top&
\mathbf 0_{d\times d}&
\mathbf 0_d^\top&
\dfrac{\gamma}{\beta}\mathbf I_{d\times d}
\end{bmatrix}.
\]
Clearly, for $\kappa>0$, $\mathbf{A_0(U)}$ is a symmetric positive definite symmetrizer of the system \eqref{hyperbolic_system}.
\end{proof}
In view of these explicit symmetrizers, we now state the first result of this article.
\begin{theorem}[Local well-posedness of smooth solutions of \eqref{hyperbolic_system}-\eqref{initial_data_hyp_relaxation}]
\label{local_wellposedness}
\leavevmode\\
Let $d\in\mathbb{N}$ and let the initial data
\[
(h_0,\psi_0,\mathbf{q}_0,w_0,\mathbf{p}_0)\in (H^m(\mathbb{R}))^{2d+3},\qquad 
m>\frac{d}{2}+1,\qquad h_0>0 \ \text{a.e.}
\]
Let the functions
$\mathcal{M},\Pi$ satisfy Assumption (\ref{assumption.A}). Then there exists $T^*\in(0,\infty)$ depending only on the initial data, for which the Cauchy problem 
associated with \eqref{hyperbolic_system} and initial data \eqref{initial_data_hyp_relaxation} admits a unique solution. Moreover, the 
solution satisfies
\[
(h^{\alpha, \tau}, \psi^{\alpha, \tau},\mathbf{q}^{\alpha, \tau}, w^{\alpha, \tau}, \mathbf{p}^{\alpha, \tau}) \in C\!\big([0,T^*];(H^m(\mathbb{R}))^{2d+3}\big)
\cap C^1\!\big([0,T^*];(H^{m-1}(\mathbb{R}))^{2d+3}\big),~h>0~\text{a.e.}
\]
\end{theorem}
\begin{proof}
In view of Lemma \ref{symmetrizer}, the hyperbolic system \eqref{hyperbolic_system}  is a Friedrichs-symmetrizable hyperbolic system, and thus the local well-posedness of classical solutions for sufficiently smooth initial data is an immediate consequence of the classical theory of the hyperbolic conservation laws \cite{dafermos2005hyperbolic}.   
\end{proof}
\begin{remark}
Note that in the proof of Lemma \ref{symmetrizer}, the only structural condition required is
\[
\alpha+\Pi'(h)>0\quad \forall\,\, h\in [0, \infty).
\] 
In particular, this condition holds for any disjoining pressure satisfying $\Pi'(h)\geq -\delta'$ for some constant $\delta'>0$ and $\alpha>\delta'$. This observation shows that the hyperbolic system \eqref{hyperbolic_system} remains locally well-posed even in situations where the disjoining pressure does not satisfy Assumption (\ref{assumption.A}). In particular, Cahn–Hilliard-type potentials \cite{elliott1996cahn} satisfy $\Pi'(h)\geq -\delta'$ for some $\delta'>0$ and thus the hyperbolic system \eqref{hyperbolic_system} is locally well-posed as an approximate model in this setting as well. This motivates us to present numerical simulations in Section \ref{sec: numerics} for Cahn–Hilliard-type equations. Although our analytical results in the following sections do not cover the case where \eqref{eq: main} involves a non-convex energy functional, the associated hyperbolic relaxation system remains locally well-posed and thus provides a meaningful approximation framework. The main challenge for the non-convex energy functional arises from the fact that there is a mismatch between the free energy variable and its gradient, and thus, the error caused by these non-convex parts of the energy can not be absorbed directly with the relaxation system of the form \eqref{hyperbolic_system} unless one includes some artificial diffusion into the dynamics. This also indicates the need to design hyperbolic relaxation systems, where the free energy variable is tied with the gradient. 
\end{remark}
\subsection{Entropy structure of the system \eqref{hyperbolic_system}}\label{sec: wellposedness}
In this section, we prove that the system \eqref{hyperbolic_system} is equipped with an entropy/entropy-flux pair with a strictly convex entropy under the Assumption (\ref{assumption.A}). 

For $\mathbf{U}\rel\in \mathcal{U}$, we first write the system \eqref{hyperbolic_system} in the balance-law form
\begin{equation}\label{balance_law_form}
    \mathbf{U}\rel_t
    +\sum_{j=1}^{d}
    \partial_{x_j}\mathbf{F}_j(\mathbf{U}\rel)
    =\mathbf{S}(\mathbf{U}\rel),
\end{equation}
where for every $j=1,\ldots,d$, flux functions are defined as
\[
\mathbf{F}_j(\mathbf{U} )
=
\left(
q_j ,
\alpha q_j ,
\frac{\Pi(h )+\psi }{\tau}\mathbf e_j,
-\frac{\gamma}{\beta}p_j ,
-w \mathbf e_j
\right)^\top.
\]
Here, $\mathbf e_j$ denotes the $j$-th canonical basis vector
of $\mathbb R^d$. Moreover, the source term in \eqref{balance_law_form} is given by
\[
\mathbf{S}(\mathbf{U} )
=
\left(
0,
-\alpha w ,
-\frac{\mathbf q }{\tau\mathcal M(h )},
\frac{\psi }{\beta},
\mathbf 0_d
\right)^\top.
\]
By an entropy/entropy-flux pair, we mean a pair 
$(E\rel, \mathbf{Q}\rel): \mathcal{U}\mapsto \mathbb{R}^{d+1}$ such that $E\rel$ is convex  with positive definite Hessian $\nabla^2_{{\mathbf U}} E\rel$ and such that for all $\mathbf{U\rel}\in \mathcal{U}$, the pair $(E\rel, \mathbf{Q}\rel)$ satisfies the compatibility condition
\begin{equation}\label{entropy}
\nabla_{\mathbf{U}} E\rel(\mathbf{U\rel})^\top\mathbf{DF_j(U\rel)}=\nabla_{\mathbf{U\rel}}Q_j\rel(\mathbf{U\rel})^\top, \, j=1, 2, \ldots, d,\,
%, \quad 
\end{equation}
and
 \begin{equation}\label{entropy_dissipation}
\dfrac{d}{dt}E\rel(\mathbf{U\rel})+\nabla\cdot \mathbf{Q}\rel(\mathbf{U\rel})\leq 0
%, \quad 
\end{equation}
in the sense of distributions.
\begin{Lemma}\label{entropy_lemma}
Let $\mathcal{M}$ and $\Pi$ satisfy the Assumption (\ref{assumption.A}). There exists an entropy/entropy flux pair $({E}^{\alpha, \tau}, \mathbf{Q}^{\alpha, \tau})$ with strictly convex entropy ${E}^{\alpha, \tau}$. It is given by
\begin{equation}\label{convex-entropy}
    E^{\alpha, \tau}(\mathbf{U}) =W(h)+\dfrac{({\psi})^2}{2\alpha}+\dfrac{\tau {|\mathbf{q}|}^2}{2}+\dfrac{\beta {w}^2}{2}+\dfrac{\gamma {|\mathbf{p}|}^2}{2}, 
\end{equation}
and
\begin{align}\label{entropy_flux}
    \mathbf{Q}^{\alpha, \tau}(\mathbf{U})= ((\Pi(h)+\psi)\mathbf{q}-\gamma \mathbf{p}w).
\end{align} 
In particular, for any $t\in (0, T]$, under periodic boundary conditions, a classical solution $\mathbf{U}\rel$ of the system \eqref{hyperbolic_system} satisfies the equation
\begin{equation}\label{energy_inequality}
\begin{aligned}
\int_{\Omega}\bigg(W(h\rel)+\dfrac{{\psi\rel}^2}{2\alpha}&+\dfrac{\tau {|\mathbf{q}\rel|}^2}{2}+\dfrac{\beta {w\rel}^2}{2}+\dfrac{\gamma {|\mathbf{p}\rel|}^2}{2}\bigg)\, d\mathbf{x}+\int_{0}^t \int_{\Omega} \dfrac{{|\mathbf{q}\rel|}^2}{\mathcal{M}(h\rel)}\, d\mathbf{x}\, ds\\
&=\int_{\Omega}\left(W(h\rel_0)+\dfrac{{\psi\rel_0}^2}{2\alpha}+\dfrac{\tau {{|\mathbf{q\rel_0}|}}^2}{2}+\dfrac{\beta {w\rel_0}^2}{2}+\dfrac{\gamma {|\mathbf{p\rel_0}|}^2}{2}\right)\, d\mathbf{x}.
\end{aligned}
\end{equation}
\end{Lemma}

\begin{proof}
In what follows, we prove that the pair $(E\rel, \mathbf{Q\rel})$ as defined in \eqref{convex-entropy}-\eqref{entropy_flux} satisfies the compatibility relation \eqref{entropy} and the dissipation identity \eqref{entropy_dissipation}. First, we compute the gradient of $E\rel$ with respect to $\mathbf U\rel$. Precisely, we have
\begin{align}\label{eq: gradient_entropy}
\nabla_{\mathbf U}E^{\alpha,\tau}(\mathbf U\rel)
=
\left(
\Pi(h\rel),
\frac{\psi\rel}{\alpha},
\tau\mathbf q\rel,
\beta w\rel,
\gamma\mathbf p\rel
\right)^\top.
\end{align}
For a fixed $j\in\{1,\ldots,d\}$, the Jacobian of the $j$-th flux computes as
\[
 \mathbf{DF_j}(\mathbf U\rel)
=
\begin{pmatrix}
0 & 0 & \mathbf e_j & 0 & \mathbf 0_d\\
0 & 0 & \alpha\mathbf e_j & 0 & \mathbf 0_d\\
\dfrac{\Pi'(h\rel)}{\tau}\mathbf e_j^\top
&
\dfrac{1}{\tau}\mathbf e_j^\top
&
\mathbf{0}_{d\times d}
&
\mathbf 0_d^\top
&
\mathbf{0}_{d\times d}
\\
0 & 0 & \mathbf 0_d & 0
&-\dfrac{\gamma}{\beta}\mathbf e_j\\
\mathbf 0_d^\top & \mathbf 0_d^\top & \mathbf{0}_{d\times d}
&-\mathbf e_j^\top&\mathbf{0}_{d\times d}
\end{pmatrix}.
\]
Thus, using \eqref{eq: gradient_entropy}, we have
\begin{align*}
\nabla_{\mathbf U}E^{\alpha,\tau}(\mathbf U\rel)^\top \mathbf{DF_j}(\mathbf U\rel)=
\left(
q_j\rel\Pi'(h\rel),
q_j\rel,
\bigl(\Pi(h\rel)+\psi\rel\bigr)\mathbf e_j,
-\gamma p_j\rel,
-\gamma w\rel\mathbf e_j
\right).
\end{align*}
Now define
\[
Q_j^{\alpha,\tau}(\mathbf U )
=
\bigl(\Pi(h )+\psi \bigr)q_j 
-\gamma p_j  w .
\]
Its gradient is
\[
\nabla_{\mathbf U}Q_j^{\alpha,\tau}(\mathbf U )
=
\left(
q_j \Pi'(h ),
q_j ,
\bigl(\Pi(h )+\psi \bigr)\mathbf e_j,
-\gamma p_j ,
-\gamma w \mathbf e_j
\right)^\top.
\]
Therefore,
\[
\nabla_{\mathbf U}E^{\alpha,\tau}(\mathbf U )^\top
 \mathbf{DF_j}(\mathbf U )
=
\nabla_{\mathbf U}Q_j^{\alpha,\tau}
(\mathbf U )^\top.
\]
Since $j\in\{1,\ldots,d\}$ was arbitrary, the entropy flux is
\[
\mathbf Q^{\alpha,\tau}(\mathbf U )
=
\bigl(\Pi(h )+\psi \bigr)\mathbf q 
-\gamma\mathbf p  w.
\]
The strict convexity of
$E^{\alpha,\tau}$ is an immediate consequence of the Assumption~(\ref{A.2}),

We now derive the entropy identity \eqref{entropy_dissipation}. Multiplying
\eqref{balance_law_form} by
$\nabla_{\mathbf U}E^{\alpha,\tau}(\mathbf U\rel)^\top$
and using the entropy compatibility condition gives
\[
\dfrac{d}{dt}E^{\alpha,\tau}(\mathbf U\rel)
+\nabla\cdot
\mathbf Q^{\alpha,\tau}(\mathbf U\rel)
=
\nabla_{\mathbf U}E^{\alpha,\tau}(\mathbf U\rel)^\top
\mathbf S(\mathbf U\rel).
\]
The contribution of the source term is
\begin{align*}
\nabla_{\mathbf U}E^{\alpha,\tau}(\mathbf U\rel)^\top
\mathbf S(\mathbf U\rel)&\quad=
\frac{\psi\rel}{\alpha}(-\alpha w\rel)
+\tau\mathbf q\rel\cdot
\left(
-\frac{\mathbf q\rel}
{\tau\mathcal M(h\rel)}
\right)
+\beta w\rel\frac{\psi\rel}{\beta}
\\
&\quad=
-\psi\rel w\rel
-\frac{|\mathbf q\rel|^2}{\mathcal M(h\rel)}
+\psi\rel w\rel
\\
&\quad=
-\frac{|\mathbf q\rel|^2}{\mathcal M(h\rel)}.
\end{align*}
We therefore obtain the local entropy identity
\begin{equation}\label{local_entropy_identity}
\dfrac{d}{dt}E^{\alpha,\tau}(\mathbf U\rel)
+\nabla\cdot
\mathbf Q^{\alpha,\tau}(\mathbf U\rel)
+\frac{|\mathbf q\rel|^2}{\mathcal M(h\rel)}
=0.
\end{equation}
In view of Assumption (\ref{A.1}), it is easy to obtain the dissipation identity \eqref{entropy_dissipation} in the sense of distributions. 

Moreover, integrating \eqref{local_entropy_identity} over $\Omega$ and using the
periodic boundary conditions give
\[
\int_\Omega
\nabla\cdot
\mathbf Q^{\alpha,\tau}(\mathbf U\rel)
\,d\mathbf x
=0.
\]
It follows that
\[
\frac{d}{dt}
\int_\Omega
E^{\alpha,\tau}(\mathbf U\rel)\,d\mathbf x
+
\int_\Omega
\frac{|\mathbf q\rel|^2}{\mathcal M(h\rel)}
\,d\mathbf x
=0.
\]
Finally, integrating over the time interval $(0,T)$, we obtain
\[
\int_\Omega
E^{\alpha,\tau}(\mathbf U\rel(\cdot, T))\,d\mathbf x
+
\int_0^T\int_\Omega
\frac{|\mathbf q\rel|^2}{\mathcal M(h\rel)}
\,d\mathbf x\,ds
=
\int_\Omega
E^{\alpha,\tau}(\mathbf U\rel_0)\,d\mathbf x,
\]
which is precisely \eqref{energy_inequality}.
\end{proof}
\begin{remark}
Note that the integral of the entropy as defined in \eqref{convex-entropy} converges to the energy \eqref{energy} of the fourth-order PDE \eqref{eq: main} as the relaxation parameter vanishes. Thus, the integrated entropy of the hyperbolic system can also be considered as the energy of the system. This implies that the hyperbolic system \eqref{hyperbolic_system} is an energy-consistent relaxation system for \eqref{eq: main}.
\end{remark}
Since the relaxation system \eqref{hyperbolic_system} is a nonlinear system of balance laws, we can't expect it to have global smooth solutions. Therefore, for the discussions to follow, we define weak solutions of the hyperbolic system \eqref{hyperbolic_system} as follows.
\begin{definition}\label{weak_soln}
Let $T>0$ and $\mathbf{U}\rel_0=\left(h_0, \psi_0, \mathbf{q}_0, w_0, \mathbf{p}_0\right)^\top\in L^\infty_{\mathrm{loc}}(\Omega;\mathbb{R}^{2d+3})$ be given.  We call a function $\mathbf{U}\rel \in L^\infty_{\mathrm{loc}}\big([0,T)\times\Omega;\,\mathbb{R}^{2d+3}\big)$ a weak solution of the system \eqref{hyperbolic_system} with initial data 
$\mathbf{U}\rel_0$, if for every vector-valued function ${\bm \varphi} \in C_0^\infty\big([0,T)\times\Omega;\,\mathbb{R}^{2d+3}\big)$, the following identity holds
\begin{equation*}
    \int_{0}^{T}\int_{\Omega} 
    \Big( \mathbf{U}\rel\cdot {\bm \varphi}_t
          + \mathbf{F}(\mathbf{U}\rel)\cdot {\bm \varphi}_x
          +\mathbf{S}(\mathbf{U}\rel)\cdot {\bm \varphi} \Big)\,d\mathbf{x}\,dt
    + \int_{\Omega} \mathbf{U}\rel_0(\mathbf{x})\cdot {\bm \varphi}(\mathbf{x}, 0)\,d\mathbf{x} = 0.
\end{equation*}
\end{definition}
Moreover, in view of explicit entropy/entropy pairs of \eqref{hyperbolic_system}, we define weak entropy solutions of \eqref{hyperbolic_system} as follows.
\begin{definition}\label{entropy_soln}
We call a function $\mathbf{U}\rel\in {L}_{loc}^{\infty}((0,T)\times \Omega; ~\mathbb{R}^{2d+3} )$ a weak  entropy solution of the system \eqref{hyperbolic_system} associated with the entropy \eqref{convex-entropy} if it is a weak solution and if 
\begin{equation}
\begin{aligned}\hspace*{-0.56cm}
    \int_{0}^{T}\int_{\Omega} \bigg(E^{\alpha, \tau}(\mathbf{U}\rel)\varphi_t+Q^{\alpha, \tau}(\mathbf{U}\rel)\varphi_x&+\nabla_{\mathbf{U}}E^{\alpha, \tau}(\mathbf{U}\rel)\cdot \mathbf{S(U\rel)}\varphi\bigg) \, d\mathbf{x}\, dt\\
    &+\int_{\Omega} \left(E^{\alpha, \tau}(\mathbf{U}\rel_0)\varphi(\mathbf{x}, 0)\right)\, d\mathbf{x}\geq 0. 
\end{aligned}
\end{equation}
holds for all non-negative functions $\varphi \in C^\infty_0([0,T)\times \Omega)$.
\end{definition}
\section{Asymptotic limits of the relaxation approximation \eqref{hyperbolic_system}}\label{sec: convergence}
In this section, we utilize the entropy structure of the approximation system \eqref{hyperbolic_system} to prove the convergence of weak entropy solutions of the hyperbolic system \eqref{hyperbolic_system} towards sufficiently regular solutions of the Cauchy problem \eqref{eq: main}-\eqref{initial_data_main}. Let us denote the solution of the approximate system \eqref{hyperbolic_system} as $\mathbf{U}\rel=(h\rel, \psi\rel, \mathbf{q}\rel, w\rel , \mathbf{p}\rel)^\top\in \mathcal{U}$ and the comparison solution by $\mathbf{{U}}=({h}, {\psi},{\mathbf{q}}, {w}, {\mathbf{p}})^\top=({h}, -\gamma \Delta{h},-\mathcal{M}({h})\nabla(\Pi({h})-\gamma \Delta{h}), {h}_t, \nabla{h})^\top\in \mathcal{U}$ with $h$ being the classical solution of \eqref{eq: main}-\eqref{initial_data_main}. Then we prove the following main result of this article.
\begin{theorem}[Convergence of solutions of the relaxation approximation]\label{main_result}
Let $T>0$ and $\Pi, \mathcal{M}\in C^\infty(\mathbb{R})$ satisfy the Assumption (\ref{assumption.A}). Further, let ${h}\in H^4((0, T)\times \Omega)$ with $\nabla \cdot(\mathcal{M}(h)\nabla(\gamma \Delta h))\in L^\infty([0, T]; \Omega)$, ${h}_{tt}\in L^2([0, T]; L^2( \Omega))$ be a solution of \eqref{eq: main}-\eqref{initial_data_main} with initial data ${h}_0\in H^8(\Omega)$. Let for each $\alpha, \beta, \tau>0$, $\mathbf{U}\rel=(h^{\alpha, \tau}, \psi^{\alpha, \tau},\mathbf{q}^{\alpha, \tau}, w^{\alpha, \tau}, \mathbf{p}^{\alpha, \tau})^\top\in \mathcal{U}$  be weak entropy solution of \eqref{hyperbolic_system} with well-prepared initial data \eqref{initial_data_hyp_relaxation}. \\
Then
    \begin{align}\label{convergence_rates}
        \lVert h^{\alpha, \tau}-{h}\rVert_{L^{\infty}([0, T]; L^2(\Omega))}+ \lVert \mathbf{p}^{\alpha, \tau}-\mathbf{\nabla}{h}\rVert_{L^{\infty}([0, T]; L^2(\Omega))}&+ \lVert \mathbf{q}^{\alpha, \tau}-\mathcal{M}({h})\nabla(\Delta {h}-\Pi({h}))\rVert_{L^{2}([0, T]; L^2(\Omega))}\nonumber\\
        &=\mathcal{O}\left(\sqrt{\dfrac{1}{\alpha}}\right)+\mathcal{O}\left(\sqrt{{\beta}}\right)+\mathcal{O}\left({\tau}\right).
    \end{align}
\end{theorem}
\begin{proof}[Proof of Theorem \ref{main_result}]
In view of Assumption (\ref{assumption.A}), the entropy $E^{\alpha, \tau}:\mathcal{U}\mapsto \mathbb{R}$ from \eqref{convex-entropy} is a uniformly convex function. Therefore, we define the associated  relative entropy $\mathcal{E}^{\alpha, \tau}:\mathcal{U}\times \mathcal{U}\mapsto \mathbb{R}$ for the approximate solution ${\mathbf U}\rel$  and the smooth function $\mathbf{{U}}:= \mathbf{U}(\mathbf{x}, t)\in \mathcal{U}$ as \\
\begin{align}\label{rel_entropy_hyp_system} \hspace*{-0.3cm}
    \mathcal{E}^{\alpha, \tau}(\mathbf{U}\rel| \mathbf{ U})&=W(h\rel)-W({h})-\Pi({h})(h\rel-{h})+\dfrac{\gamma (\mathbf{p}\rel-\mathbf{{p}})^2}{2}\nonumber\\
    &\quad+\dfrac{ \beta(w\rel-{w})^2}{2}+\dfrac{(\psi\rel-{\psi})^2}{2\alpha}+\dfrac{\tau(\mathbf{q\rel-{q}})^2}{2}.
\end{align}
Moreover, the uniform convexity of $E\rel$ allows us to get the bound \\
\begin{align}\label{relative_entropy_relation_main}
    C\bigg((h\rel-{h})^2+\dfrac{1}{\alpha}(\psi\rel-{\psi})^2+\beta(w\rel-{w})^2&+(\mathbf{p}\rel-{\mathbf{p}})^2+\tau (\mathbf{q}\rel-\mathbf{{q}})^2\bigg)\nonumber\\
    &\leq \mathcal{E}^{\alpha, \tau}(\mathbf{U}\rel|\mathbf{{U}}),
\end{align}
where $C>0$ is an $\alpha-$ and $\tau$-independent constant which depends only on the initial data.

In order to prove the theorem, we first rewrite the limit equation \eqref{eq: main} as 
\begin{equation}\label{reformulated_thin_film}
\begin{aligned}
{h}_t+\nabla\cdot \mathbf{{q}}&=0,\\
    \dfrac{1}{\alpha}{\psi}_t+\nabla\cdot \mathbf{{q}}&=-{w}-\dfrac{\gamma}{\alpha}  \partial_t \Delta{h},\\
    \tau \mathbf{{q}}_t+\nabla(\Pi({h})+{\psi})&=-\dfrac{\mathbf{{q}}}{\mathcal{M}({h})}-\tau \left(\mathcal{M}({h})\nabla(\Pi({h})-\gamma \Delta h)\right)_t,\\
    \beta {w}_t-\gamma\nabla\cdot \mathbf{{p}}&={\psi}+\beta {h}_{tt},\\
    \mathbf{{p}}_t-\nabla{w}&=0.
\end{aligned}
\end{equation}
We observe that $\mathbf{{U}}=({h}, {\psi},{\mathbf{q}}, {w}, {\mathbf{p}})^\top=({h}, -\gamma \Delta{h},-\mathcal{M}({h})\nabla(\Pi({h})-\gamma \Delta{h}), {h}_t, \nabla{h})^\top$ is a classical solution of \eqref{reformulated_thin_film} with well-prepared initial data (see \eqref{initial_data_hyp_relaxation})
\begin{equation}\label{initial_data_hyp}
\mathbf U_0=\mathbf U(\cdot,0)=
\bigl(h_0,-\gamma\Delta h_0,-\mathcal M(h_0)\nabla(\Pi(h_0)-\gamma\Delta h_0),\partial_t h(\cdot,0),\nabla h_0\bigr)^\top.
\end{equation}
It is then easy to show 
that for  
the entropy/entropy-flux  pair $(E^{\alpha, \tau}, \mathbf{Q}^{\alpha, \tau})$
from Lemma \ref{entropy_lemma}, the function ${\mathbf U}$ satisfies  the integral identity 
\begin{equation}
\begin{aligned} \hspace*{-0.3cm}
    \int_{0}^{T}\int_{\Omega} \bigg(E^{\alpha, \tau}(\mathbf{{U}})\varphi_t+\sum_{j=1}^{d}Q^{\alpha, \tau}_j(\mathbf{{U}})\varphi_{x_j}+\nabla_{\mathbf{{U}}}E^{\alpha, \tau}(\mathbf{{U}})\cdot &\mathbf{{S}({U})\varphi}\big) \, d\mathbf{x}\, dt\\
    &+\int_{\Omega} E^{\alpha, \tau}(\mathbf{{U}}_0)\varphi( x,0)\, d\mathbf{x}= 0\label{entrop_identity_dispersion}
\end{aligned}
\end{equation}
 for all $\varphi \in C_{0}^{1}([0,T)\times\Omega)$, $\varphi \ge 0$.
In \eqref{entrop_identity_dispersion} we used  
\[
\mathbf{{S}({{\mathbf {{U}}}})}=\bigg(0, -\alpha{w}- \gamma \partial_t(\Delta{h}),  -\dfrac{\mathbf{q}}{\tau \mathcal{M}({h})}-\ \left(\mathcal{M}({h})\nabla(\Pi({h})-\gamma \Delta h)\right)_t,{\beta}^{-1}{\psi}+{h}_{tt}, 0\bigg)^\top.
\]
Moreover, since $\mathbf{U}\rel$ is a weak entropy solution of  \eqref{hyperbolic_system}, \eqref{initial_data_hyp_relaxation}, we have by Definition \ref{entropy_soln}
\begin{equation}\label{rel_entropy_identity_hyp_system}
\begin{aligned}
%\begin{array}{rcl}
% \lefteqn{ \hspace*{-5cm}
\int_{0}^{T}\int_{\Omega} \Bigg(E^{\alpha, \tau}(\mathbf{U}\rel)\varphi_t+\sum_{j=1}^{d}Q^{\alpha, \tau}_j(\mathbf{U}\rel){\varphi_{x_j}}&+\nabla_{\mathbf{U}}E^{\alpha, \tau}(\mathbf{U}\rel)\cdot \mathbf{S(U\rel)}\varphi\Bigg) \, d\mathbf{x}\, dt\\
  &+\int_{\Omega} \left(E^{\alpha, \tau}(\mathbf{U}\rel_0)\varphi(x,0)\right)\, d\mathbf{x}\geq 0. 
 %   \end{array}
 \end{aligned}
\end{equation}
Now, we subtract \eqref{entrop_identity_dispersion} from \eqref{rel_entropy_identity_hyp_system} to obtain the following integral relation
\begin{equation}\label{difference_idenity}
\begin{aligned}
    &\hspace*{-2cm}\int_{0}^{T}\int_{\Omega} \left(E^{\alpha, \tau}(\mathbf{U}\rel)-E^{\alpha, \tau}(\mathbf{{U}}))\varphi_t+\sum_{j=1}^{d}(Q_j^{\alpha, \tau}(\mathbf{U}\rel)-Q^{\alpha, \tau}_j(\mathbf{{U}}))\varphi_{x_j}\right) \, d\mathbf{x}\, dt\\
+&\int_{0}^{T}\int_{\Omega}\bigg(\nabla_{\mathbf{U}}E^{\alpha, \tau}(\mathbf{U}\rel)\cdot \mathbf{S(U\rel)}-\nabla_{\mathbf{{U}}}E^{\alpha, \tau}(\mathbf{{U}})\cdot \mathbf{{S}({U})}\bigg)\varphi \, d\mathbf{x}\, dt\\
+&\int_{\Omega} \left(E^{\alpha, \tau}(\mathbf{U}_0^{\alpha, \tau})-E^{\alpha, \tau}(\mathbf{{U}}_0)\right)\varphi(\mathbf{x}, 0)\, d\mathbf{x}\geq 0. 
\end{aligned}
\end{equation}
Moreover, since $\mathbf{{U}}$ is a classical solution of \eqref{reformulated_thin_film}, it must also be a distributional solution of the system \eqref{reformulated_thin_film}. Thus, if  we  choose  for  $\varphi \in C^\infty_0([0,T) \times \Omega )$ the test function ${\bm \varphi}=\varphi  \nabla_{\mathbf{U}} E^{\alpha, \tau}(\mathbf{{U}})\in  C^1_0([0,T) \times \Omega; \mathbb{R}^{2d+3})$ like in Definition \ref{weak_soln}, we obtain from the fact that $\mathbf{U}\rel$ is a weak solution of \eqref{hyperbolic_system}-\eqref{initial_data_hyp_relaxation}, the identity
\begin{equation}\label{weak_relative}
\begin{aligned}
    &\int_{0}^{T}\int_{\Omega} (\mathbf{U}\rel-\mathbf{{U}})\cdot (\varphi \nabla_{\mathbf{U}} E^{\alpha, \tau}(\mathbf{{U}}))_t+\sum_{j=1}^{d}(\mathbf{F}_j\mathbf{(U\rel)}-\mathbf{F}_j\mathbf{({U})}\cdot(\varphi\nabla_{\mathbf{U}} E^{\alpha, \tau}(\mathbf{{U}}))_{x_j}\, d\mathbf{x}\, dt\\
    &\qquad\qquad\qquad\qquad+\int_{0}^{T}\int_{\Omega} (\mathbf{S(U\rel)}-\mathbf{{S}({U})})\cdot (\varphi \nabla_{\mathbf{U}} E^{\alpha, \tau}(\mathbf{{U}})) \, d\mathbf{x}\, dt \\
    &\qquad\qquad\qquad\qquad+\int_{\Omega} (\mathbf{U}\rel_0-\mathbf{{U}}_0)\cdot (\varphi\nabla_{\mathbf{U}} E^{\alpha, \tau}(\mathbf{{U}}))( x, 0)\, d\mathbf{x}= 0. 
\end{aligned}
\end{equation}
Let $\nabla_{\mathbf U}^2 E^{\alpha,\tau}$ denote the Hessian matrix of
$E^{\alpha,\tau}$. Then, by the chain rule,
\[
\partial_t\big(\nabla_{\mathbf U}E^{\alpha,\tau}(\mathbf U)\big)
=
\nabla_{\mathbf U}^2E^{\alpha,\tau}(\mathbf U)\,\partial_t\mathbf U,
\qquad
\partial_{x_j}\big(\nabla_{\mathbf U}E^{\alpha,\tau}(\mathbf U)\big)
=
\nabla_{\mathbf U}^2E^{\alpha,\tau}(\mathbf U)\,\partial_{x_j}\mathbf U,
\quad j=1,\dots,d.
\]
Thus, we get after  subtracting \eqref{weak_relative} from \eqref{difference_idenity}  the identity 
\begin{equation}\label{difference_idenity_weak_form}
\begin{aligned}
    &I:=\int_{0}^{T}\int_{\Omega} \left(\mathcal{E}^{\alpha, \tau}(\mathbf{U}\rel|\mathbf{{U}}))\varphi_t+\sum_{j=1}^{d}(\mathcal{Q}^{\alpha, \tau}_j(\mathbf{U}\rel|\mathbf{{U}}))\varphi_{x_j}\right) \, d\mathbf{x}\, dt\\
&~-\underbrace{\int_{0}^{T}\int_{\Omega} \left(\big(\mathbf{U}\rel-\mathbf{{U}}\big)\cdot\big( \nabla^2_{\mathbf{{U}}} E^{\alpha, \tau}(\mathbf{U}\big)\mathbf{{U}}_t) +\sum_{j=1}^{d}\big(\mathbf{F}_j\mathbf{(U\rel)}-\mathbf{F}_j\mathbf{({U})})\cdot(\nabla^2_\mathbf{U} E^{\alpha, \tau}(\mathbf{{U}}) \mathbf{{U}}_{x_j}\big)\right) \varphi\, d\mathbf{x}\, dt}_{=:I_1}\\
&~+\underbrace{\int_{0}^{T}\int_{\Omega}\left(\nabla_{\mathbf{U}}E^{\alpha, \tau}(\mathbf{U}\rel)-\nabla_{\mathbf{U}}E^{\alpha, \tau}(\mathbf{{U}})\right)\cdot \mathbf{S(U\rel)} \varphi\, d\mathbf{x}\, dt}_{=:I_2}+\int_{\Omega} \left(\mathcal{E}^{\alpha, \tau}(\mathbf{U}_0\rel|\mathbf{{U}}_0)\right)\varphi(x,0)\, d\mathbf{x}\geq 0. 
\end{aligned}
\end{equation}
Here we used the relative entropy flux components
\[
\mathcal{Q}^{\alpha, \tau}_j(\mathbf{U}\rel| {\mathbf U})= Q^{\alpha, \tau}_j(\mathbf{U}\rel)-Q^{\alpha, \tau}_j(\mathbf{{U}})-\nabla_{\mathbf{U}}E^{\alpha, \tau}(\mathbf{{U}})\cdot (\mathbf{F}_j(\mathbf{U}\rel)-\mathbf{F}_j\mathbf{({U})}), \]
which satisfy for an $\alpha$-independent constant $C\ge0$ the inequality
$|\mathcal{Q}_j^{\alpha, \tau}(\mathbf{U}\rel| {\mathbf U})|\leq C |\mathcal{E}^{\alpha, \tau}(\mathbf{U}\rel| {\mathbf U})|.
$
Now, we focus on the integrals $I_1$ and $I_2$ in
\eqref{difference_idenity_weak_form}. First, consider $I_1$.
Using the strong form of \eqref{reformulated_thin_film}, that is,
\[
\mathbf U_t+\sum_{j=1}^{d}\mathbf {DF_j}(\mathbf U)\mathbf U_{x_j}
=
\mathbf S(\mathbf U),
\]
we obtain
\begin{equation}\label{I_1}
\begin{aligned}
I_1
&=
\int_0^T\int_\Omega
\sum_{j=1}^{d}
\Big(
\mathbf F_j(\mathbf U\rel)-\mathbf F_j(\mathbf U)-\mathbf {DF_j}(\mathbf U)(\mathbf U\rel-\mathbf U)
\Big)\cdot
\big(\nabla_{\mathbf U}^2E^{\alpha,\tau}(\mathbf U)\mathbf U_{x_j}\big)
\,\varphi\,d\mathbf x\,dt\\
&\quad
+\underbrace{\int_0^T\int_\Omega
\big(\nabla_{\mathbf U}^2E^{\alpha,\tau}(\mathbf U)(\mathbf U\rel-\mathbf U)\big)\cdot
\mathbf S(\mathbf U)
\,\varphi\,d\mathbf x\,dt}_{=:I_3},
\end{aligned}
\end{equation}
where we used the symmetry of the Hessian matrix.\\
Now, using the explicit expressions for the fluxes $\mathbf F_j$ and the entropy
$E^{\alpha,\tau}$, one obtains using a straightforward calculation that
\begin{equation}\label{taylor_remainder_F}
\begin{aligned}
&\sum_{j=1}^{d}
\Big(
\mathbf F_j(\mathbf U\rel)-\mathbf F_j(\mathbf U)-\mathbf {DF_j}(\mathbf U)(\mathbf U\rel-\mathbf U)
\Big)\cdot
\big(\nabla_{\mathbf U}^2E^{\alpha,\tau}(\mathbf U)\mathbf U_{x_j}\big)\\
&\qquad=
\big(\Pi(h\rel)-\Pi(h)-\Pi'(h)(h\rel-h)\big)\,\nabla\cdot \mathbf q.
\end{aligned}
\end{equation}
Moreover, $I_2$ from \eqref{difference_idenity_weak_form} and $I_3$ from \eqref{I_1} together become
\begin{equation}
    \begin{aligned}
        I_2-I_3&=-\nabla^2_{\mathbf{U}}E^{\alpha, \tau}(\mathbf{{U}})(\mathbf{U\rel}-\mathbf{{U}}))\cdot \mathbf{{S}({U})}+\left(\nabla_{\mathbf{U}}E^{\alpha, \tau}(\mathbf{U}\rel)-\nabla_{\mathbf{U}}E^{\alpha, \tau}(\mathbf{{U}})\right)\cdot \mathbf{S(U\rel)}\\
        &=\dfrac{\gamma}{\alpha} (\psi\rel-{\psi})\partial_t\Delta{h}-\beta (w\rel-{w}){h}_{tt}\\
   &-(\mathbf{q\rel-{q}})\cdot \left(\dfrac{\mathbf{{q\rel}}}{\mathcal{M}(h\rel)}-\dfrac{\mathbf{{q}}}{\mathcal{M}({h})}\right)+\tau (\mathbf{q\rel-{q}})\cdot \left(\mathcal{M}({h})\nabla(\Pi({h})+{\psi})\right)_t.
    \end{aligned}
\end{equation}
Now $I$ in \eqref{difference_idenity_weak_form} reduces to  
\begin{equation}\label{difference_idenity_weak_form_simplified}
\begin{aligned}
\hspace*{-1cm}I=&    \int_{0}^{T}\int_{\Omega} \left(\mathcal{E}^{\alpha, \tau}(\mathbf{U}\rel|\mathbf{{U}})\varphi_t+\sum_{j=1}^{d}(\mathcal{Q}^{\alpha, \tau}_j(\mathbf{U}\rel|\mathbf{{U}}))\varphi_{x_j}\right) \, d\mathbf{x}\, dt\\
&-\int_{0}^{T}\int_{\Omega} \big(\Pi(h\rel)-\Pi({h})-\Pi'({h})(h\rel-{h})\big)
     \big(\nabla\cdot \mathbf{q}\big)\varphi\, d\mathbf{x}\, dt\\
&+\int_{0}^{T}\int_{\Omega}\left(\dfrac{\gamma}{\alpha} (\psi\rel-{\psi})\partial_t\Delta{h}-\beta (w\rel-{w}){h}_{tt}\right)\varphi\, d\mathbf{x}\, dt\\
&-\int_{0}^{T}\int_{\Omega}\left((\mathbf{q\rel-{q}})\cdot \left(\dfrac{\mathbf{{q\rel}}}{\mathcal{M}(h\rel)}-\dfrac{\mathbf{{q}}}{\mathcal{M}({h})}\right)\right)\varphi\, d\mathbf{x}\, dt\\
   &+\int_{0}^{T}\int_{\Omega}\left(\tau (\mathbf{q\rel-{q}})\cdot \left(\mathcal{M}({h})\nabla(\Pi({h})+{\psi})\right)_t\right)\varphi\, d\mathbf{x}\, dt\\
   &+\int_{\Omega} \mathcal{E}^{\alpha, \tau}(\mathbf{U}_0\rel|\mathbf{{U}}_0)\varphi(x,0)\, d\mathbf{x}\geq 0. 
\end{aligned}
\end{equation}
Since we are using periodic boundary conditions, we get for the test function $\varphi(x,t)=\vartheta(t)\in C_{0}^{1}([0, T))$ the derivatives  $\vartheta_t=\vartheta'(t)$ and $\vartheta_x=0$. In such a situation, the integral identity \eqref{difference_idenity_weak_form_simplified} simplifies to
\begin{equation}\label{difference_idenity_periodic}
\begin{aligned}
    &\hspace*{-0.2cm}\int_{0}^{T}\int_{\Omega} \mathcal{E}^{\alpha, \tau}(\mathbf{U}\rel| {\mathbf U})\vartheta'\, d\mathbf{x}\, dt+\int_{\Omega} \left(\mathcal{E}^{\alpha, \tau}(\mathbf{U}_0\rel| {\mathbf U}_0)\right)\vartheta(\cdot, 0)\, d\mathbf{x}\\
&\geq \int_{0}^{T}\int_{\Omega} \big(\Pi(h\rel)-\Pi({h})-\Pi'({h})(h\rel-{h})\big)
     \big(\nabla\cdot \mathbf{q}\big)\vartheta\, d\mathbf{x}\, dt\\
& -\int_{0}^{T}\int_{\Omega}\left(\dfrac{\gamma}{\alpha} (\psi\rel-{\psi})\partial_t\Delta{h}-\beta (w\rel-{w}){h}_{tt}+(\mathbf{q\rel-{q}})\cdot \left(\dfrac{\mathbf{{q\rel}}}{\mathcal{M}(h\rel)}-\dfrac{\mathbf{{q}}}{\mathcal{M}({h})}\right)\right)\vartheta\, d\mathbf{x}\, dt\\
   &-\int_{0}^{T}\int_{\Omega}\left(\tau (\mathbf{q\rel-{q}})\cdot \left(\mathcal{M}({h})\nabla(\Pi({h})+{\psi})\right)_t\right)\vartheta\, d\mathbf{x}\, dt. 
\end{aligned}
\end{equation}
Following the proof of Theorem 5.2.1 in \cite{dafermos2005hyperbolic}, we fix the test function $\vartheta$ for some time $\sigma\in (0, T)$. The test function $\vartheta$ for some $\epsilon>0$ is given by
\begin{align}
  \vartheta(t)
  &= 
  \begin{cases}
    1, & 0 \le t < \sigma,\\[0.3em]
    \dfrac{\sigma - t}{\epsilon} + 1, & \sigma \le t \le \sigma + \epsilon,\\[0.3em]
    0, & t \ge \sigma + \epsilon,
  \end{cases} 
\end{align}
such that 
\begin{align}
  \vartheta'(t)
  &= 
  \begin{cases}
    0, & 0 \le t < \sigma,\\[0.3em]
    -\dfrac{1}{\epsilon}, & \sigma < t < \sigma + \epsilon,\\[0.3em]
    0, & t \ge \sigma + \epsilon.
  \end{cases}
\end{align}
Clearly, \eqref{difference_idenity_periodic} then reduces to 
\begin{equation}\label{difference_idenity_periodic_updated}
\begin{aligned}
    \hspace*{-0.2cm}\dfrac{1}{\epsilon}\int_{\sigma}^{\sigma+\epsilon}\int_{\Omega} \mathcal{E}^{\alpha, \tau}(\mathbf{U}\rel| {\mathbf U})\, d\mathbf{x}\, dt\leq&\int_{\Omega} \mathcal{E}^{\alpha, \tau}(\mathbf{U}_0\rel| {\mathbf U}_0)\, d\mathbf{x}-\underbrace{\int_{0}^{\sigma}\int_{\Omega}\mathcal{R}\vartheta(t)\, d\mathbf{x}\, dt}_{=:I_4}-\underbrace{\int_{\sigma}^{\sigma+\epsilon}\int_{\Omega}\mathcal{R}\vartheta(t)\, d\mathbf{x}\, dt}_{=:I_5},
\end{aligned}
\end{equation}
where
\begin{align}\label{reamainder_R}
\mathcal{R}&=\big(\Pi(h\rel)-\Pi({h})-\Pi'({h})(h\rel-{h})\big)
     \big(\nabla\cdot \mathbf{q}\big)-\left(\tau (\mathbf{q\rel-{q}})\cdot \left(\mathcal{M}({h})\nabla(\Pi({h})+{\psi})\right)_t\right)\nonumber\\
     -&\dfrac{\gamma}{\alpha} (\psi\rel-{\psi})\partial_t\Delta{h}+\beta (w\rel-{w}){h}_{tt}-(\mathbf{q\rel-{q}})\cdot \left(\dfrac{\mathbf{{q\rel}}}{\mathcal{M}(h\rel)}-\dfrac{\mathbf{{q}}}{\mathcal{M}({h})}\right).
\end{align}     
Now for $I_4$, in view of Young's inequality, we have
\begin{equation}\label{I_4}
\begin{aligned}
    I_4&\leq {\lVert \nabla\cdot \mathbf{q}\rVert}_{L^{\infty}((0, \sigma)\times \Omega)} \int_{0}^{\sigma}\int_{\Omega}\left|\big(\Pi(h\rel)-\Pi({h})-\Pi'({h})(h\rel-{h})\big)\right|\, d\mathbf{x}\,dt\\
&+\int_{0}^{\sigma}\int_{\Omega}\left(\dfrac{1}{2\alpha}(\psi\rel-{\psi})^2+\dfrac{\beta}{2} (w\rel-{w})^2+\dfrac{1}{2} (\mathbf{q}\rel-{\mathbf{q}})^2\right)\,d\mathbf{x}\,dt\\
&+\int_{0}^{\sigma}\int_{\Omega}\left(\dfrac{\gamma^2}{2\alpha}(\partial_t\Delta{h})^2+\dfrac{\beta}{2} {h}_{tt}^2+ \dfrac{\tau^2 \left(\mathcal{M}({h})\nabla(\Pi({h})+{\psi})\right)_t^2}{2}\right)\,d\mathbf{x}\,dt\\
&-\underbrace{\int_{0}^{\sigma}\int_{\Omega}(\mathbf{q\rel-{q}})\cdot \left(\dfrac{\mathbf{{q\rel}}}{\mathcal{M}(h\rel)}-\dfrac{\mathbf{{q}}}{\mathcal{M}({h})}\right)\, d\mathbf{x}\, dt}_{=I_6}.
\end{aligned}
\end{equation}
In view of Assumption (\ref{A.2}) we have ${\lVert \Pi''\rVert}_{\infty}\leq L$ and thus
\begin{align}\label{eq: taylor_f}
\big(\Pi(h\rel)-\Pi({h})-\Pi'({h})(h\rel-{h})\big)\leq \dfrac{L}{2}(h\rel-{h})^2.
\end{align}
Moreover, the monotonocity of  $\Pi$ with $\Pi'(h)\geq \delta>0$ ensures
\begin{align}\label{rel_ent_2}
\dfrac{L}{2}(h\rel-{h})^2\leq \dfrac{L}{2\delta}\big(W(h\rel)-W({h})-\Pi({h})(h\rel-{h})\big).
\end{align}
For $I_6$, we have, by Young's inequality, the Lipschitz bound on $\mathcal{M}(h)$ and the lower bound on $\mathcal{M}(h)$ the estimate
\begin{align}\label{q_estimates}
    -\int_{0}^{\sigma}\int_{\Omega}(\mathbf{q\rel-{q}})&\cdot \left(\dfrac{\mathbf{{q\rel}}}{\mathcal{M}(h\rel)}-\dfrac{\mathbf{{q}}}{\mathcal{M}({h})}\right)\, d\mathbf{x}\, dt \nonumber\\
    &=-\int_{\Omega} \dfrac{1}{\mathcal{M}(h\rel)}|\mathbf{q\rel-{q}}|^2\, d\mathbf{x}-\int_{\Omega}\dfrac{(\mathcal{M}(h\rel)-\mathcal{M}({h}))}{\mathcal{M}(h\rel)\mathcal{M}({h})}\mathbf{{q}}\cdot(\mathbf{q\rel}-\mathbf{{q}})\, d\mathbf{x}, \nonumber\\
    &\leq-\int_{\Omega} \dfrac{1}{\mathcal{M}(h\rel)}|\mathbf{q\rel-{q}}|^2\, d\mathbf{x}+\dfrac{L\lVert\mathbf{{q}}\rVert_{\infty}}{m_0^2}\int_{\Omega} |h\rel-{h}|(\mathbf{q\rel}-\mathbf{{q}})\, d\mathbf{x}, \nonumber\\
    &\leq-\int_{\Omega} \left(\dfrac{1}{\mathcal{M}(h\rel)}-m_0\right)|\mathbf{q\rel-{q}}|^2\, d\mathbf{x}+\dfrac{L^2(\lVert\mathbf{{q}}\rVert_{\infty})^2}{m_0^3}\int_{\Omega} (h\rel-{h})^2\, d\mathbf{x}\nonumber\\
    &\leq C\big(\lVert\mathbf{{q}}\rVert_{\infty}, L, m_0\big) \left(W(h\rel)-W({h})-\Pi({h})(h\rel-{h})\right).
\end{align}
Therefore, using \eqref{eq: taylor_f}, \eqref{rel_ent_2} and \eqref{q_estimates} in \eqref{I_4} and in view of the relative entropy from \eqref{rel_entropy_hyp_system}, we have
\begin{align*}
   \hspace{-1 cm} I_4&\leq C\left(L,{\lVert \mathbf{q}\rVert}_{L^{\infty}}, \delta, m_0\right)  {\lVert \nabla\cdot \mathbf{q}\rVert}_{L^{\infty}((0, \sigma)\times \Omega)}\int_{0}^{\sigma}\int_{\Omega}\mathcal{E}^{\alpha, \tau}(\mathbf{U\rel}|\mathbf{{U}})\, d\mathbf{x}\,dt\\
&+\int_{0}^{\sigma}\int_{\Omega}\left(\dfrac{\gamma^2}{2\alpha}(\partial_t\Delta{h})^2+\dfrac{\beta}{2} {h}_{tt}^2+ \dfrac{\tau^2 \left(\mathcal{M}({h})\nabla(\Pi({h})+{\psi})\right)_t^2}{2}\right)\,d\mathbf{x}\,dt,
\end{align*}
where the constant $C\left(L,{\lVert \mathbf{q}\rVert}_{L^{\infty}}, \delta, m_0\right)$ does not depend on the parameters $\alpha$ and $\beta$.\\
Furthermore, for $I_5$ in \eqref{difference_idenity_periodic_updated}, one can apply the Cauchy-Schwarz inequality to obtain the inequality
\begin{align*}
    I_5\leq \sqrt{\epsilon |\Omega| }{\lVert\mathcal{R}\rVert}_{{L}^{2}((\sigma, \sigma+\epsilon)\times \Omega)}\rightarrow 0, \quad \text{as}~\epsilon\rightarrow 0,
\end{align*}
where $\mathcal{R}$ is defined in \eqref{reamainder_R}.

Thus, by letting $\epsilon \rightarrow 0$ in \eqref{difference_idenity_periodic_updated}, we obtain the inequality
\begin{align*}
   \hspace{-0.2 cm} \int_{\Omega} \mathcal{E}^{\alpha, \tau}(\mathbf{U}\rel
    (x, \sigma)| {\mathbf U}(x, \sigma))\, d\mathbf{x}&\leq\int_{\Omega} \mathcal{E}^{\alpha, \tau}(\mathbf{U}_0\rel| {\mathbf U}_0)\, d\mathbf{x}\\
   & \quad+ C\left(L,\lVert \mathbf{q}\rVert_{L^{\infty}}, \delta, m_0\right) {\lVert \nabla\cdot \mathbf{q}\rVert}_{L^{\infty}((0, \sigma)\times \Omega)}\int_{0}^{\sigma}\int_{\Omega}\mathcal{E}^{\alpha, \tau}(\mathbf{U\rel}|\mathbf{{U})}\, d\mathbf{x}\,dt\\
&\quad+\int_{0}^{\sigma}\int_{\Omega}\left(\dfrac{\gamma^2}{2\alpha}(\partial_t\Delta{h})^2+\dfrac{\beta}{2} {h}_{tt}^2+ \dfrac{\tau^2 \left(\mathcal{M}({h})\nabla(\Pi({h})+{\psi})\right)_t^2}{2}\right)\,d\mathbf{x}\,dt.
\end{align*}
Therefore, arguing as in Theorem 5.2.1 in \cite{dafermos2005hyperbolic} and using Gronwall's inequality, one can conclude that
\begin{equation}
\begin{aligned}
 &\hspace{-0.3 cm}\sup_{t\in (0, T)}  \int_{\Omega}\mathcal{E}^{\alpha, \tau}(\mathbf{U}\rel|\mathbf{{U}})\, d\mathbf{x}\\
 &\leq \exp\left(CT{\lVert \nabla\cdot \mathbf{q}\rVert}_{L^{\infty}((0, \sigma)\times \Omega)}\right)\bigg(\int_{\Omega} \mathcal{E}^{\alpha, \tau}(\mathbf{U}_0\rel|\mathbf{{U}}_0)\, d\mathbf{x}+\dfrac{\gamma^2}{2\alpha} {\lVert {\partial_t \Delta h}\rVert}_{{{L}^{2}}(\Omega_T)}^2\\
 &\qquad\qquad \qquad+\dfrac{\beta}{2} {\lVert{h}_{tt}\rVert}_{{{L}^{2}}(\Omega_T)}^2
 + \tau^2{\lVert {\partial_t \nabla \Delta h}\rVert}_{{{L}^{2}}(\Omega_T)}^2 \bigg).
\end{aligned}
\end{equation}
Since we consider well-prepared initial data of the form \eqref{initial_data_hyp_relaxation}, we have $\mathbf{U}\rel_0=\mathbf{{U}}_0$ and thus
\begin{align*}
\int_{\Omega}&\mathcal{E}^{\alpha, \tau}(\mathbf{U}_0\rel|\mathbf{{U}}_0)\,d\mathbf{x}\\
&=W(h_0\rel)-W({h}_0)-\Pi({h}_0)(h_0\rel-{h}_0)+\dfrac{\gamma}{2}(p\rel_0-{p}_0)^2+\dfrac{ \beta}{2}  (w\rel_0-{w}_0)^2+\dfrac{1}{2\alpha}(\psi\rel_0-{\psi}_0)^2=0.
\end{align*}
Then if $\alpha\rightarrow \infty$ and $\tau\rightarrow 0$, we must have $\int_{\Omega}\mathcal{E}^{\alpha, \tau}(\mathbf{U}\rel|\mathbf{{U}})\, d\mathbf{x}\rightarrow 0$. This implies that the well-prepared initial data of the form \eqref{initial_data_hyp_relaxation} leads to the convergence of the solutions of the approximate system \eqref{hyperbolic_system} to the classical solutions of \eqref{eq: main}, \eqref{initial_data_main} as $\alpha\rightarrow \infty$ and $\tau\rightarrow 0$. The convergence rates in \eqref{convergence_rates} are trivial in view of \eqref{relative_entropy_relation_main}.
The theorem is then proved.
\end{proof}
\section{Numerical results}\label{sec: numerics}
In this section, we present a series of numerical experiments covering various physical applications. We begin with a test case for the thin-film equation with zero disjoining pressure. To demonstrate the broad applicability of our relaxation approach, we consider the relaxation approximation for the Cahn–Hilliard equations with both constant and degenerate mobilities and a non-monotone disjoining pressure. Finally, we consider fourth-order equations where strong nonlinear effects are present in the flux function, further illustrating the versatility and scope of the proposed approximation framework.
For the sake of simplicity, we keep our discussion to one dimension only. However, one can easily extend the approach used here to multiple dimensions.
\subsection{Numerical scheme for the limit equation \eqref{eq: main}}
We solve, in one dimension on the computational domain $\Omega= [x_L,x_R]$ with $N$ equally spaced points
$x_i=x_L+(i-1)\dx$, the general equation
\begin{equation}
	h_t + \partial_x F(h)= A_h(\Pi(h)-\gamma h_{xx}), \qquad \text{with} \quad  A_h(\mu):=\big(M(h) \mu_x\big)_x ,
	\label{eq:original_num}
\end{equation}
where an optional convective flux $F$ is added. The pressure term is split, following Eyre \cite{eyre1998unconditionally} 
\begin{equation*}
	\Pi(h) = S h + (\Pi(h) - S h), \qquad S \geq \max \Pi'(h).
\end{equation*}
A classic mass-conservative discretization of the operator $A_h$ writes 
\begin{equation*}
	\prn{A \mu}_i = \frac{1}{\dx^2}\Big[
	M_{i+\hlf}\,(\mu_{i+1}-\mu_i)-M_{i-\hlf}\,(\mu_{i}-\mu_{i-1})\Big],
	\qquad M_{i\pm\hlf}=\tfrac12\big(M_i+M_{i\pm1}\big) 
\end{equation*}
The values $\prn{h_{xx}}_i$ are computed using the second-order finite difference $(Lh)_i=(h_{i+1}-2h_i+h_{i-1})/\dx^2$.
Combining these terms to discretize \eqref{eq:original_num}, taking the convex and fourth-order derivative implicitly and keeping the concave and advection parts explicitly yields
\begin{equation*}
		h_i^{n+1}
		-\dt\,S\,(A^n h^{n+1})_i
		+\gamma\,\dt\,(A^n Lh^{n+1})_i = h_i^n+\dt\,(A^n\prn{\Pi(h^n)-S\,h^n})_i
		-\frac{\dt}{\dx}\big(\hat F^n_{i+\hlf}-\hat F^n_{i-\hlf}\big),
\end{equation*}
where $A^n = A(h^n)$. The implicit operator is then solved by a banded LU factorization.  
The state-dependent mobility is treated by a frozen-coefficient Picard iteration. The
convective term, when present, is added explicitly through a Rusanov flux with MUSCL reconstruction.
\subsection{Numerical scheme for the hyperbolic system}
For the computational domain $\Omega_c=[x_L, x_R]$, the cell centers and interfaces of the $i$th cell are given by
\[
x_i = x_L + \gh{i-\frac{1}{2}}\dx, \qquad
x_{i+1/2}=\dfrac{1}{2}(x_i+x_{i+1}).
\]
Moreover, the time is discretized such that $t^{n+1}=t^n+\Delta t$ with $t^0=0$
and $\Delta t$ is determined using a CFL condition for the hyperbolic system
\eqref{hyperbolic_system}. We denote by $\mathbf{U}_i^n$ the cell average of the
numerical solution over $\Omega_i$ at time $t^n$.
In order to solve the system \eqref{hyperbolic_system_new} numerically, we first
cast it in compact form as
\begin{equation}
	\mathbf{U}^{\alpha,\tau}_t + \mathbf{f}(\mathbf{U}^{\alpha,\tau})_x
	= \mathbf{S}^{\alpha,\tau}(\mathbf{U}^{\alpha,\tau}),
	\label{eq:system_numerics}
\end{equation}
with $\mathbf{U}^{\alpha,\tau}=(h,\mathbf{V})^{T}$, where
$\mathbf{V}=(\psi,q,w,p)^{T}$ collects the film thickness and the relaxation
variables, $\mathbf{f}$ is the first-order flux and
$\mathbf{S}^{\alpha,\tau}$ the stiff relaxation source. We use the
semi-implicit Lax--Wendroff scheme introduced in~\cite{barthwal2025relaxation},
to which we refer for its derivation and properties. In short, the first-order
hyperbolic operator is advanced by an explicit two-step Lax--Wendroff method,
chosen for its low numerical dissipation and well suited to the fast
characteristic speeds of the system. The stiff relaxation source is treated
implicitly, all the more so as additional stiffness is introduced by the mobility function. In Richtmyer form, the intercell
states are first predicted at the half-time level,
\begin{align}
	\mathbf{U}_{i\pm\hlf}^{n+\hlf} = \hlf\prn{\mathbf{U}_{i}^n + \mathbf{U}_{i\pm1}^n}
	\mp \frac{\Delta t}{2\Delta x}\prn{\mathbf{f}(\mathbf{U}_{i\pm1}^n)-\mathbf{f}(\mathbf{U}_{i}^n)}
	+ \frac{\Delta t}{2}\,\mathbf{S}\left(\mathbf{U}_{i\pm\hlf}^{n+\hlf}\right).
	\label{eq:LW_step1}
\end{align}
Owing to the block structure of the system, this implicit step can be solved
without any nonlinear iteration. Splitting the unknowns as
$\mathbf{U}=(h,\mathbf{V})^T$, the mass balance for $h$ is sourceless and
therefore updates explicitly,
\begin{align}
	h_{i\pm\hlf}^{n+\hlf} = \hlf\prn{h_{i}^n + h_{i\pm1}^n}
	\mp \frac{\Delta t}{2\Delta x}\prn{\mathbf{f}_1(\mathbf{U}_{i\pm1}^n)-\mathbf{f}_1(\mathbf{U}_{i}^n)}.
	\label{eq:LW_step1_h}
\end{align}
Once $h_{i\pm\hlf}^{n+\hlf}$ is known, the inverse mobility
$\mathcal{M}\left(h_{i\pm\hlf}^{n+\hlf}\right)^{-1}$ becomes a known coefficient, so the
source acting on $\mathbf{V}$ is linear and the remaining equations
of~\eqref{eq:LW_step1} can be inverted directly,
\begin{align}
	\mathbf{V}_{i\pm\hlf}^{n+\hlf} = \frac{1}{2}\prn{\mathbf I - \frac{\dt}{2}\mathbf{S}_{\mathbf{V}}}^{-1}
	\prn{\prn{\mathbf{V}_{i}^n + \mathbf{V}_{i\pm1}^n}
		\mp \frac{\Delta t}{\Delta x}\prn{\mathbf{f}_{\mathbf{V}}(\mathbf{U}_{i\pm1}^n)-\mathbf{f}_{\mathbf{V}}(\mathbf{U}_{i}^n)}},
\end{align}
where $\mathbf{S}_{\mathbf{V}}=\mathbf{S}_{\mathbf{V}}\left(h_{i\pm\hlf}^{n+\hlf}\right)$ is
the now constant source matrix of the $\mathbf{V}$-subsystem.
\subsection{Test 1: Thin film in near-rupture regime}
\label{ssec:deadcore}
We first consider the initial value problem for a lubrication model \cite{grun2000nonnegativity} given by
\begin{equation*}
	\begin{cases}
		h_t = \prn{h^2 h_{xxx}}_{x},\\[2pt]
		h(x,0) = h_0 + (x - x_0)^4,
	\end{cases}
\end{equation*}
where $h_0=10^{-3}$ and $x_0 = 0.5$. In one dimension, the PDE is equipped with the boundary conditions \cite{grun2000nonnegativity}
\begin{equation}
	h_x = h_{xxx} = 0, \qquad \text{at} \quad x=0 \ \text{and}\  x=1.
	\label{eq:BC_TEST2}
\end{equation}
The film moves along the wall and a zone where $h$ becomes close to zero in the center before slowly progressing towards equilibrium. For the mobility power exponent $n=2$, the solution remains strictly positive, and no vacuum zone is created. 
For the hyperbolic system, in order to reproduce this test, ghost states are considered on each side of the boundary, and we set $q_{\text{gst}} = 0, p_{\text{gst}} = 0$, which mimics \eqref{eq:BC_TEST2}. First-order boundary extrapolation is used for the remaining variables.
The computational domain for both models is $\Omega_c=[0,1]$ discretized over $N=1000$ cells, and $\gamma$ is set to~$1$. The hyperbolic reformulation uses $\lambda=10^{4}$, $\tau=10^{-4}$ and $\beta=10^{-7}$, and a CFL number of $0.9$ while the original model uses a fixed time-step of $\Delta t = 10^{-3}$. The final simulation time is $t=20$.
Figure~\ref{fig:deadcore-h} compares the numerical film thicknesses $h$ and $h^{\alpha,\tau}$ for the two models. It is shown that the hyperbolic model solution $h^{\alpha,\tau}$ tracks the reference $ h$ very well throughout the thinning process, where $h$ reaches values close to zero.  
\begin{figure}[htbp]
	\centering
	\includegraphics[width=0.49\textwidth]{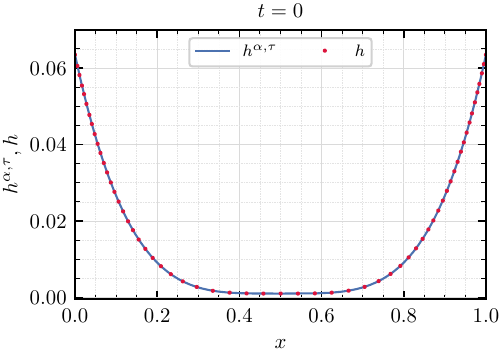}
	\includegraphics[width=0.49\textwidth]{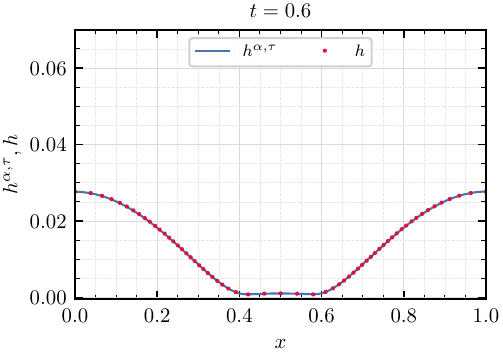}\\
	\includegraphics[width=0.49\textwidth]{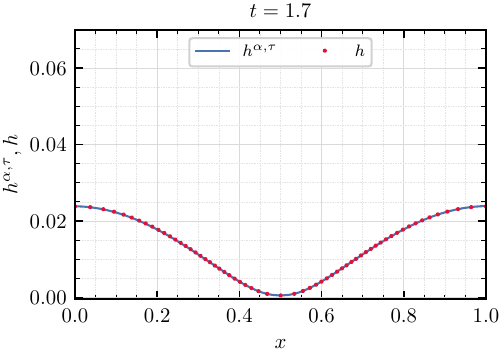}
	\includegraphics[width=0.49\textwidth]{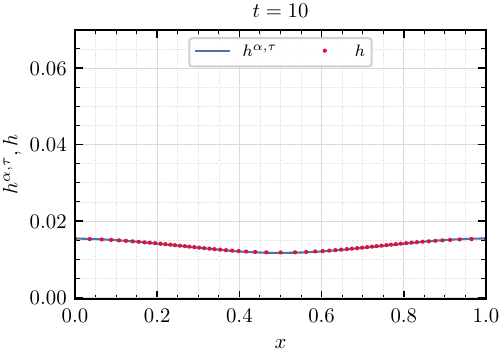}
	\caption{Test 1 -- Film thickness for the hyperbolic model
		(Blue line, $h^{\alpha,\tau}$) and the original model (red markers, reference $h$)
		at times $t=\{0,0.6,1.7,10\}$ (top-left to bottom-right).}
	\label{fig:deadcore-h}
\end{figure}

As for Test~1, we further examine the auxiliary variables.
Figure~\ref{fig:deadcore-pq} shows $p^{\alpha,\tau}$ and the flux
$q^{\alpha,\tau}$ at several times, compared with the corresponding reference
quantities. The relaxed variables again match the reference solution.
\begin{figure}[htbp]
	\centering
	\begin{subfigure}[t]{0.48\textwidth}
		\includegraphics[width=\linewidth]{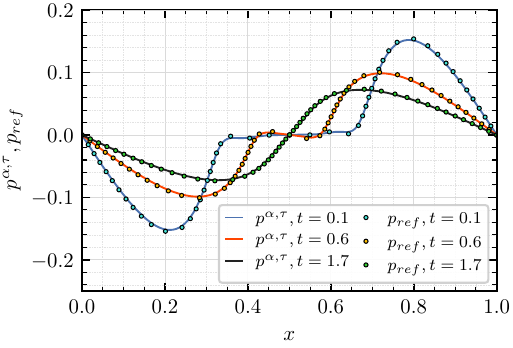}
		\caption{Pressure $p^{\alpha,\tau}$ vs.\ reference $p$.}
		\label{fig:dc-p}
	\end{subfigure}
	\begin{subfigure}[t]{0.48\textwidth}
		\includegraphics[width=\linewidth]{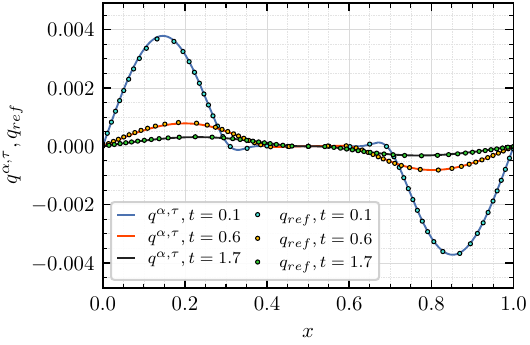}
		\caption{Flux $q^{\alpha,\tau}$ vs.\ reference $q$.}
		\label{fig:dc-q}
	\end{subfigure}
	\caption{Test 1 -- Auxiliary variables $p^{\alpha,\tau},q^{\alpha,\tau}$ of the hyperbolic model compared with the reference solution at times $t=\{0.1,0.6,1.7\}$.}
	\label{fig:deadcore-pq}
\end{figure}
It is instructive in this case to compare the energy decay of both models.
Since $\Pi\equiv 0$, the original free energy density reduces to the interfacial
contribution $\tfrac12\gamma h_x^2$, which tends to zero
as the film flattens towards its equilibrium state.
Figure~\ref{fig:deadcore-energy} reports the energy evolution on a
logarithmic scale for three values of the relaxation parameter
$\alpha\in\{10^{2},10^{3},10^{4}\}$. All curves decay monotonically and, in the
early and intermediate regimes, follow the reference $F$ closely. However, for the hyperbolic system, energy does not converge to zero, but to a constant state, dependent on the relaxation parameters. Indeed, in view of \eqref{energy_inequality}, the hyperbolic system decays until $q^{\rel}$ reaches zero. However, the energy carries additional contributions from the variables $\psi\rel$ and $w\rel$ that do not vanish exactly at equilibrium. 
Indeed, at a flat state in every variable, one obtains for instance for $\psi$ the ODE $\psi\rel_{tt}=-\tfrac{\alpha}{\beta}\psi\rel$, so $(\psi\rel,w\rel)$ do not come to necessarily to rest but may oscillate about the origin at frequency $\sqrt{\alpha/\beta}$. The film profile
$h\rel$ does reach an equilibrium, but this does not represent a genuine steady state of the full
hyperbolic system. The figure shows that stiffer relaxation results in better tracking of the original energy decay over longer intervals. 
\begin{figure}[htbp]
	\centering
	\includegraphics[width=0.6\textwidth]{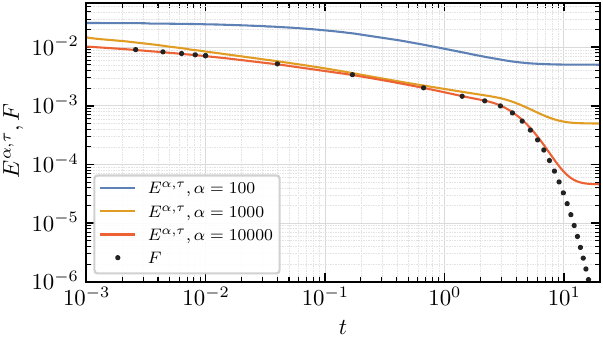}
\caption{Test 1 -- Free energy decay on a log-log scale:
	hyperbolic energy $E^{\alpha,\tau}$ for relaxation parameters
	$\alpha\in\{10^{2},10^{3},10^{4}\}$ (solid lines) compared with the original
	free energy $F$ (black markers). $\beta = 10^{-3}\alpha$ and $\tau = 10^-4$ is fixed. As $\alpha$ increases, $E^{\alpha,\tau}$
	converges towards $F$; for small $\alpha$ the energy saturates at a finite
	value set by the relaxation terms.}
\label{fig:deadcore-energy}
\end{figure}
\subsection{Test 2: Ostwald ripening in binary mixtures}\label{ssec:ostwald}
We now consider the initial value problem for the Cahn--Hilliard equation ($\Pi(c)\equiv c^3-c$)
\begin{equation}\label{eq:ch-ostwald}
	\begin{cases}
		c_t = \prn{M(c)\prn{c^3 - c- \gamma\, c_{xxx}}_x }_{x},\\[2pt]
		c(x,0) = 1 + \displaystyle\sum_{i=1}^{2}
		\prn{\tanh\prn{\dfrac{x-x_i-r_i}{\sqrt{2\gamma}}}
			-\tanh\prn{\dfrac{x-x_i+r_i}{\sqrt{2\gamma}}}},
	\end{cases}
\end{equation}
where $x_1=0.30$, $x_2=0.75$ and $r_1=0.12$, $r_2=0.06$ are the positions and
the radii of two neighboring one-dimensional bubbles. Starting from this already phase-separated but unevenly distributed state, the smaller bubble progressively diffuses and is absorbed by the larger one, which subsequently becomes stationary, a phenomenon known as Ostwald ripening. This test is done in the case of a constant mobility $M(c)=1$ as well as for a degenerate mobility $M(c)=1-c^2$.
For both models, we take a computational domain $\Omega_c = [0,1]$, discretized over $N=1000$ cells and we set $\gamma=10^{-3}$. The hyperbolic reformulation additionally
uses the relaxation parameters $\lambda=10^{3}$, $\tau=10^{-4}$ and
$\beta=10^{-7}$. The hyperbolic system uses a CFL number of $0.95$, whereas the original model uses a fixed time step $\Delta t=10^{-4}$.
\subsubsection*{Test 2A: Constant mobility $M=1$}
In this case, the final simulation time is $t=0.4$, which leaves a margin after equilibrium is reached.
Figure~\ref{fig:ostwald-c} compares the concentration obtained with the
hyperbolic model (blue solid line $c\rel$) and the original model
(red markers), at initial time, at $t=0.12$ where the dynamics is most important, and at $t=0.4$ after equilibrium is reached. The two solutions are
in excellent agreement over the whole evolution. With the
chosen relaxation parameters, no observable delay in the dynamics nor a shift in space are introduced by the hyperbolic reformulation.
\begin{figure}[htbp]
	\centering
	\begin{subfigure}[b]{0.32\textwidth}
		\includegraphics[width=\linewidth]{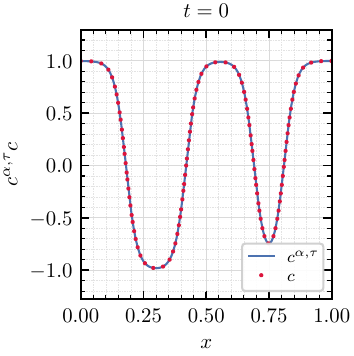}
\label{fig:ostwald-c-t0}
	\end{subfigure}
	\begin{subfigure}[b]{0.32\textwidth}
		\includegraphics[width=\linewidth]{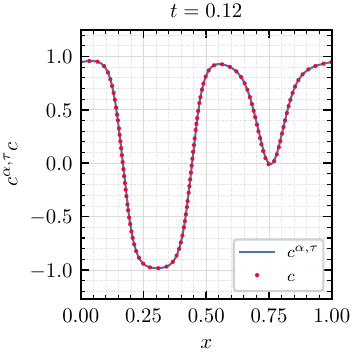}
\label{fig:ostwald-c-t02}
	\end{subfigure}
	\begin{subfigure}[b]{0.32\textwidth}
		\includegraphics[width=\linewidth]{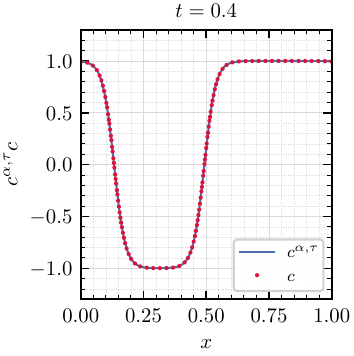}
\label{fig:ostwald-c-t04}
	\end{subfigure}
	\vspace{-0.75cm}
	\caption{Test 2A -- Comparison of the concentration $c$ for the hyperbolic
		model (blue solid line) and the original Cahn--Hilliard model (red
		markers) at times $t=\{0,0.2,0.4\}$.}
	\label{fig:ostwald-c}
\end{figure}
In order to further validate the relaxation processes in the model, we also compare in Figure~\ref{fig:ostwald-pqE} the variables $p^{\alpha,\tau}$ and $q^{\alpha,\tau}$ against their reference counterparts i.e. $p_{ref}=c_x$ and $q_{ref}=-(c^3-c-\gamma c_{xx})_x$, computed from the numerical solution of the original model using finite differences. We provide this comparison at $t=0.12$, in order to also check for eventual time delays when the solution is most dynamic. Lastly, we also plot the time evolution of total energy in both models. The figure shows that the relaxed variables $p^{\alpha,\tau}$ and $q^{\alpha,\tau}$
reproduce the reference quantities accurately, confirming that the hyperbolic
system captures not only the concentration but also the gradient and flux
structure of the underlying dynamics. Finally, the discrete energy $E^{\alpha,\tau}$ of the hyperbolic model closely follows the free energy $F$ of the original model defined in \eqref{energy} and decreases monotonically, up to reaching a constant state corresponding to equilibrium.
\begin{figure}[htbp]
	\centering
	\begin{subfigure}[t]{0.32\textwidth}
\includegraphics[width=\linewidth]{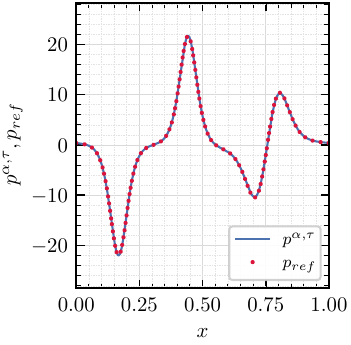}
		\caption{$p^{\alpha,\tau}$ vs.\ reference $c_x$.}
		\label{fig:ostwald-p}
	\end{subfigure}
	\begin{subfigure}[t]{0.32\textwidth}
		\includegraphics[width=1\linewidth]{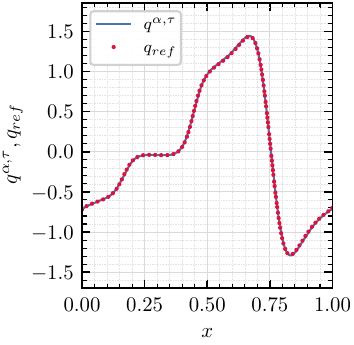}
		\caption{$q^{\alpha,\tau}$ vs.\ reference.}
		\label{fig:ostwald-q}
	\end{subfigure}
	\begin{subfigure}[t]{0.32\textwidth}
			\includegraphics[width=0.98\linewidth]{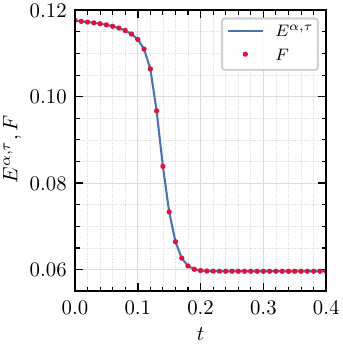}
		\caption{Energy decay
			$E^{\alpha,\tau}$ vs $F$.}
		\label{fig:ostwald-energy}
	\end{subfigure}
	\caption{Test 2A -- (a-b) Auxiliary variables of the hyperbolic
		model compared with their reference solution at $t=0.12$. (c) Plot of the total energy over time for both the hyperbolic and original models with constant mobility $M=1$.}
	\label{fig:ostwald-pqE}
\end{figure}
\subsubsection*{Test 2B: Degenerate mobility $M=1-c^2$}
The main change here with respect to the previous case is the time scale of the dynamics, since the mobility function penalizes time evolution far from the interfaces. Snapshots of the solution are given in Figure~\ref{fig:ostwaldM-c} and auxiliary variables together with the energy are provided in Figure~\ref{fig:ostwaldM-pqE}. In particular, we see here that for the same configuration and the same values of the relaxation parameters, $q^{\alpha,\tau}$ shows a slight discrepancy with the reference value, computed from the original model's numerical solution. 
\begin{figure}[htbp]
	\centering
	\begin{subfigure}[b]{0.32\textwidth}
		\includegraphics[width=\linewidth]{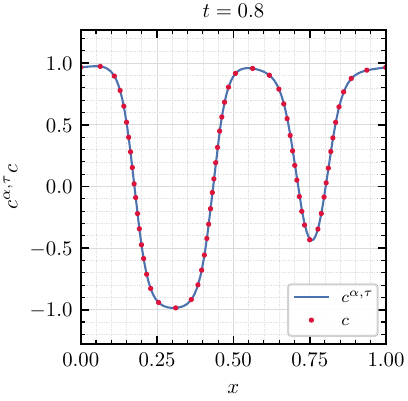}
\label{fig:ostwaldM-c-t0}
	\end{subfigure}
	\begin{subfigure}[b]{0.32\textwidth}
		\includegraphics[width=\linewidth]{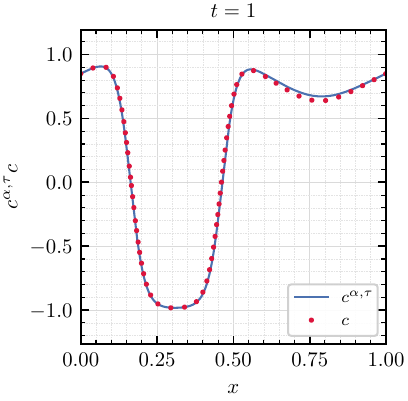}
\label{fig:ostwaldM-c-t02}
	\end{subfigure}
	\begin{subfigure}[b]{0.32\textwidth}
		\includegraphics[width=\linewidth]{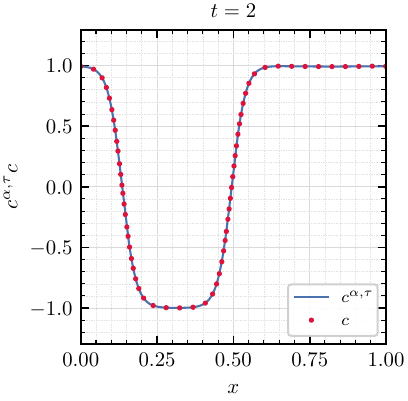}
\label{fig:ostwaldM-c-t04}
	\end{subfigure}
	\vspace{-0.75cm}
	\caption{Test 2B -- Comparison of the concentration for the hyperbolic
		(blue solid line, $c\rel$) and original Cahn--Hilliard models (red
		markers, $c$), with degenerate mobility function $M(c)=1-c^2$ at times $t=\{0.8,1,2\}$.}
	\label{fig:ostwaldM-c}
\end{figure}
\begin{figure}[htbp]
	\centering
	\begin{subfigure}[t]{0.32\textwidth}
\includegraphics[width=\linewidth]{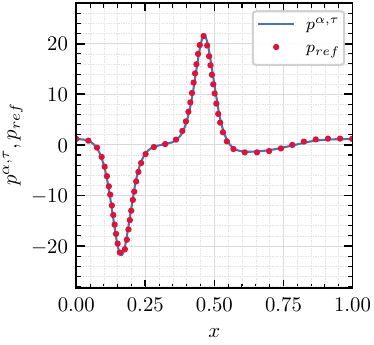}
		\caption{$p^{\alpha,\tau}$ vs.\ reference $c_x$.}
		\label{fig:ostwaldM-p}
	\end{subfigure}
	\begin{subfigure}[t]{0.32\textwidth}
		\includegraphics[width=1\linewidth]{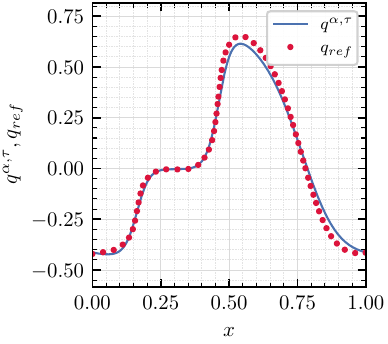}
		\caption{$q^{\alpha,\tau}$ vs.\ reference.}
		\label{fig:ostwaldM-q}
	\end{subfigure}
	\begin{subfigure}[t]{0.32\textwidth}
			\includegraphics[width=0.98\linewidth]{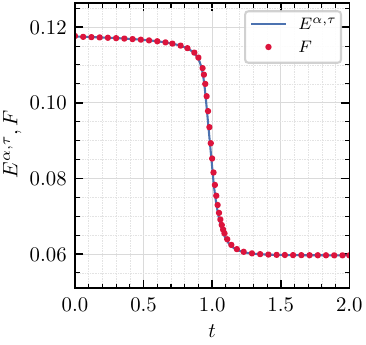}
		\caption{Energy decay
			$E^{\alpha,\tau}$ vs $F$.}
		\label{fig:ostwaldM-energy}
	\end{subfigure}
	\caption{Test 2B -- (a-b) Auxiliary variables of the hyperbolic
		model compared with their reference solution at $t=1.0$. (c) Plot of the total energy over time for both the hyperbolic and the original models with degenerate mobility $M(c)=1-c^2$.}
	\label{fig:ostwaldM-pqE}
\end{figure}

\subsection{Test 3: Undercompressive capillary shock for thin films with convection}\label{ssec:shock}
We consider here a driven thin film exhibiting an
undercompressive (capillary) shock, following the classical setting of
Bertozzi \emph{et al.}~\cite{bertozzi1999undercompressive}. The governing equations require here an additional convection term 
\begin{equation*}
	h_t + \prn{h^2 - h^3}_x + \prn{h^3 h_{xxx}}_x = 0,
\end{equation*}
$\gamma$ is set to unity. The computational domain is $\Omega_c=[0,80]$, discretized over $N=2400$ cells.
The initial condition is a $\tanh$-type front connecting the far-field thickness
$h_\infty$ to the substrate value $b$
\begin{equation}
	h(x,0) = \frac{h_\infty + b}{2} - \frac{h_\infty - b}{2}
	\tanh\prn{a\prn{x - x_0}},
    	\label{eq:IC_tanh}
\end{equation}
with the values $h_\infty=0.3$, $b=0.1$, $x_0=5$ and $a=3$. The main dynamics of the solution depend on the values of the parameters. In particular, for $b<h_\infty<h_1$, where $h_1$ is a threshold value, the solution travels and develops a capillary ridge followed by an undercompressive shock profile \cite{bertozzi1999undercompressive}. The shock speed $s$ is obtained from the
Rankine--Hugoniot relation of the convective part,
\begin{equation*}
	s = \frac{f(h_\infty) - f(b)}{h_\infty - b}
	= h_\infty  + b - (h_\infty^2+b h_\infty + b^2).
\end{equation*}
For the hyperbolic reformulation, we simply supply the convection term to the $h^{\alpha,\tau}$ equation. We use for this test case $\lambda=10^{3}$, $\tau=10^{-3}$ and
$\beta=10^{-6}$, with a CFL number of $0.95$. The final simulation time is $t=240$. Figure~\ref{fig:shock-evol} shows, on the left, the evolution of the film
profile $h^{\alpha,\tau}$ for the hyperbolic model together with the reference solution $h$, computed numerically at successive times $t\in\{0,80,160,240\}$. We highlight in the $x-$axis the theoretical positions $x_0 + s\,t_i$ predicted by the shock speed $s$, which further validates the numerical solutions. A \textit{capillary ridge} develops and grows over time until reaching a stationary profile. 
Figure~\ref{fig:shock-bertozzi} also shows a comparison of the computed profile against the reference data of Bertozzi \emph{et al.},  (digitized manually from~\cite{bertozzi1999undercompressive} and
 translated in space to the used frame). In both panels the hyperbolic solution $h^{\alpha,\tau}$ and the reference $h$ remain in close agreement: the position,
height and shape of the capillary ridge and of the undercompressive shock are recovered accurately.
\begin{figure}[htbp]
	\centering
	\begin{subfigure}[t]{0.52\textwidth}
		\includegraphics[width=\linewidth]{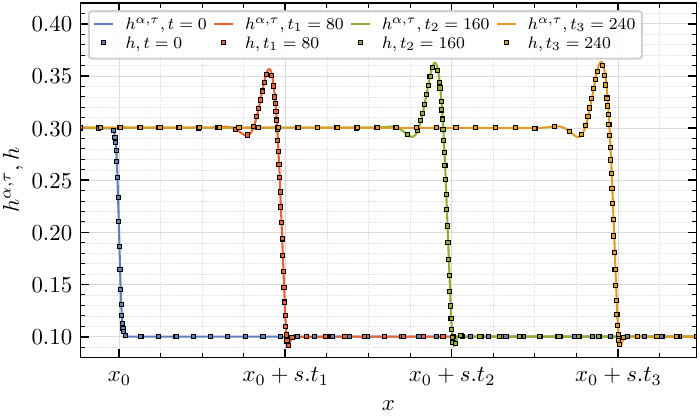}
		\caption{Travelling-wave profiles at $t\in\{0,80,160,240\}$,  each~shifted
			by $x_0 + s\,t_i$.}
		\label{fig:shock-tw}
	\end{subfigure}
	\begin{subfigure}[t]{0.47\textwidth}
		\includegraphics[width=\linewidth]{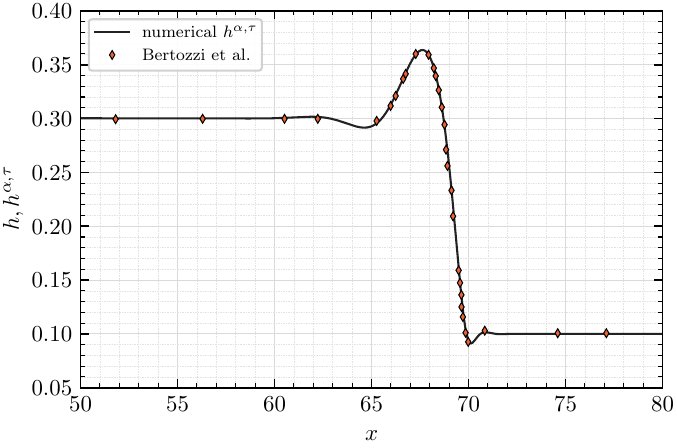}
		\caption{Computed profile vs.\ reference data of Bertozzi \emph{et al.}
			near the shock.}
		\label{fig:shock-bertozzi}
	\end{subfigure}
	\caption{Test 3 -- Undercompressive capillary shock. Film thickness for the
		hyperbolic model ($h^{\alpha,\tau}$, lines) and the reference ($h$, markers):
		(a) snapshots of the numerical solutions of the hyperbolic and original  models at different times (b)
		quantitative comparison with a reference result from Bertozzi \emph{et al.} \cite{bertozzi1999undercompressive}}
	\label{fig:shock-evol}
\end{figure}

In a similar context, we illustrate the dependence of the shock dynamics on the left film thickness $h_{\infty}$ in a test case adapted from \cite{li2011numerical}. We take a smaller computational domain $\Omega_c=[0,20]$, discretized over $N=500$ points for both models. We consider again the initial condition \eqref{eq:IC_tanh} and we take three distinct far-field heights $h_\infty\in\{0.1,\,0.2,\,0.3\}$, connecting to the same substrate value $b=0.05$. The final simulation time is set to $t=40$ and all the remaining parameters remain the same as above.
Figure~\ref{fig:shock_diff_speed} reports the three profiles $h^{\alpha,\tau}$ (lines) against their reference solutions $h$~(markers). Since the shock speed
$s=h_\infty+b-(h_\infty^2+b\,h_\infty+b^2)$ increases with $h_\infty$ (up to $h_\infty = \tfrac12(1-b)$ \cite{li2011numerical}) , the three fronts travel at increasing speeds $s_1<s_2<s_3$; the corresponding theoretical positions $x_0+s_i\,t$ are marked on the $x$-axis. In every case, the hyperbolic model with height $h^{\alpha,\tau}$ reproduces the reference solution behavior and features
accurately. 
\begin{figure}[H]
\centering
\includegraphics[width=0.7\linewidth]{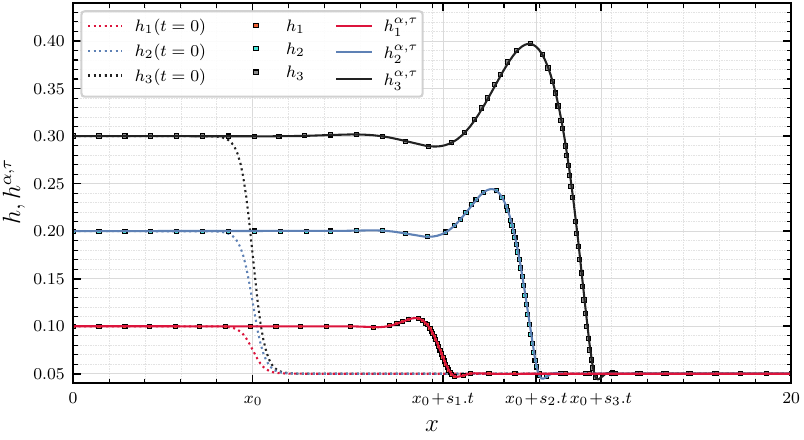}
\caption{Test 3 -- Undercompressive shock for three values of $h_\infty\in\{0.1,0.2,0.3\}$ with fixed $b=0.05$,
	at $t=40$ on $\Omega_c=[0,20]$. Markers: reference $h$; continuous lines: hyperbolic model
	$h^{\alpha,\tau}$; pointed lines: initial data. The marked abscissae $x_0+s_i\,t$ are the theoretical shock positions computed from the Rankine-Hugoniot relation.}
    \label{fig:shock_diff_speed}
\end{figure}

\subsection{Test 4: Undercompressive double shock solution}\label{ssec:shock2}
As a final validation, meant to investigate the long-time behavior of the numerical solutions, we consider another smoothed Riemann problem from Bertozzi \emph{et al.}~\cite{bertozzi1999undercompressive}, in which the initial data resolves into a double shock solution: a leading undercompressive shock and a trailing Lax shock. The film is governed by the same convection--capillarity model as in Test~3,
\begin{equation*}
	h_t + \prn{h^2 - h^3}_x = \prn{h^3 h_{xxx}}_x ,
\end{equation*} 
with mobility $M(h)=h^{3}$ and $\gamma=1$. The initial data is again \eqref{eq:IC_tanh} with $h_\infty=0.4$, $b=0.1$ and $a=20$. Since the solution structure takes a long time before the Lax-shock profile stabilizes in shape and simply travels to the left, we consider here simulations in the reference frame. Numerically, this amounts to adding a convective term, suppressing the traveling wave speed from the equations so that we solve numerically 
\begin{gather*}
	h_t + \prn{h^2 - h^3 - s.h}_x = \prn{h^3 h_{xxx}}_x,  \\
		\mathbf{U}^{\alpha,\tau}_t + \mathbf{f}(\mathbf{U}^{\alpha,\tau})_x - s.\mathbf{U}^{\alpha,\tau}_x
		= \mathbf{S}^{\alpha,\tau}(\mathbf{U}^{\alpha,\tau}),
\end{gather*}
for the original and hyperbolic systems, respectively. 
For both systems, the computational domain is $\Omega_c=[-100,100]$,
discretized over $N=5000$ cells. We use $\Delta t = 10^{-1}$ for the original model. For the hyperbolic
reformulation we use $\lambda=10^{3}$, $\tau=10^{-3}$ and $\beta=10^{-6}$, with a CFL number of $0.95$.
Figure~\ref{fig:shock2} reports the film profile at successive times
$t\in\{0,t_1,t_2,t_3\}$ in the co-moving frame $x-s\,t$, where $s=0.278647$ is the speed of the leading undercompressive shock obtained from the Rankine--Hugoniot
relation of the convective part. Plotted this way, the late-time profiles collapse onto a single travelling-wave shape, confirming that the structure propagates at constant speed. The hyperbolic solution $h^{\alpha,\tau}$ (lines)
and the reference solution $h$ (markers) are in close agreement across all the plateaus, the capillary ridge and the two connecting waves, in all the stages of the solution.
\begin{figure}[htbp]
	\centering
	\includegraphics[width=0.9\textwidth]{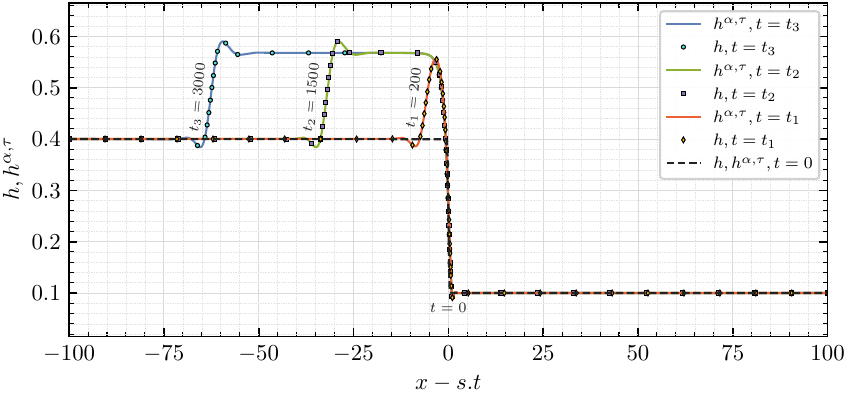}
	\caption{Test 4 -- Double shock structure. Film
		thickness for the hyperbolic model ($h^{\alpha,\tau}$, lines) and the original model ($h$, markers) at times $t\in\{0,t_1=200,t_2=1500,t_3=3000\}$, shown in the co-moving frame $x-s\,t$. Initial data is shown in black dashed lines.}
	\label{fig:shock2}
\end{figure}

\section{Conclusions and future scope}\label{sec: conclusions}
In this article, we proposed a novel energy-consistent hyperbolic system approximating the solution of complex nonlinear fourth-order partial differential equations possessing a gradient flow structure. In particular, our approximate system is designed to preserve the energy of the underlying system when the approximation parameters vanish. By utilizing the relative energy framework, we explicitly proved the convergence of weak entropy solutions of the relaxation approximation to the smooth solutions of the fourth-order PDEs. We presented a series of numerical test cases that validate our approximation approach.

The numerical scheme developed for the approximate system is not positivity preserving and thus can not be applied to more critical test cases of thin film flows. Therefore, a natural extension would be to develop a provably positivity preserving and energy stable numerical scheme for the approximation system, which will be tackled in a forthcoming paper. Another natural extension of this relaxation approach is to apply it to nonlinear equations possessing a non-convex energy, such as the degenerate Cahn-Hilliard equation, where the relaxation approach needs modifications accordingly. Another extension is to apply this relaxation approach to more complex coupled systems, such as shallow water-thin film equations or Coupled Navier-Stokes-Cahn-Hilliard systems.
%\end{comment}
\bibliographystyle{mystyle}
\bibliography{references}

\end{document}